\documentclass[reqno]{amsart} 
\usepackage{amsfonts, amsmath, amsthm, amssymb, latexsym, graphicx}

\def\b{\mathbb }

\theoremstyle{plain}
\newtheorem{theorem}{Theorem}[section]
\newtheorem{corollary}[theorem]{Corollary}
\newtheorem{lemma}[theorem]{Lemma}
\newtheorem{proposition}[theorem]{Proposition}
\theoremstyle{definition}
\newtheorem{definition}[theorem]{Definition}
\newtheorem{remark}[theorem]{Remark}

\theoremstyle{remark}
\newtheorem*{remarks}{Remarks}

\numberwithin{equation}{section}

\newcommand{\R}{\mathbb{R}}

\begin{document}

\title[Limit theorems of freezing matrix models via dual polynomials]{Limit theorems and soft edge of freezing random matrix models via dual orthogonal polynomials}

\author{Sergio Andraus}
\address{Department of Physics, Graduate School of Science, The University of Tokyo, Hongo 7-3-1,
Bunkyo-Ku, Tokyo 113-0033, Japan}
\email{sergio.andraus@phys.s.u-tokyo.ac.jp}
\author{Kilian Hermann}
\author{Michael Voit}
\address{Fakult\"at Mathematik, Technische Universit\"at Dortmund,
          Vogelpothsweg 87,
          D-44221 Dortmund, Germany}
\email{kilian.hermann@math.tu-dortmund.de, michael.voit@math.tu-dortmund.de}



\date{\today}

\begin{abstract}
$N$-dimensional Bessel and Jacobi processes  describe
 interacting particle systems with $N$ particles  and are related to
 $\beta$-Hermite, $\beta$-Laguerre, and $\beta$-Jacobi ensembles.
For fixed  $N$ there exist associated weak limit theorems (WLTs) in the freezing regime
$\beta\to\infty$ in the  $\beta$-Hermite and $\beta$-Laguerre case by 
Dumitriu and  Edelman (2005) with explicit formulas for the covariance matrices
$\Sigma_N$ in terms of the zeros of associated orthogonal polynomials. Recently,
the authors derived these WLTs
in a different way
and computed  $\Sigma_N^{-1}$ with formulas for the  eigenvalues and eigenvectors of
$\Sigma_N^{-1}$ and thus of $\Sigma_N$.
In the present paper we use these data and the theory of finite dual orthogonal 
polynomials of de Boor and  Saff to derive formulas for  $\Sigma_N$ from  
$\Sigma_N^{-1}$
where, for $\beta$-Hermite and  $\beta$-Laguerre ensembles, our formulas are simpler 
than those of  Dumitriu and  Edelman. We  use  these
 polynomials to derive asymptotic results for  
the soft edge in the freezing regime for
 $N\to\infty$ in terms of the Airy function.
 For  $\beta$-Hermite ensembles, our limit expressions are different from 
 those of Dumitriu and  Edelman.
\end{abstract}

\maketitle

\section{Introduction} 

Interacting Calogero-Moser-Sutherland particle systems  
 on $\mathbb R$ or  $[0,\infty[$  with $N$ particles 
 are described via multivariate 
Bessel processes  on closed
Weyl chambers  in $\mathbb R^N$. These have been widely studied in the mathematical and physical literature, in particular due to their connections to random matrix theory; see \cite{Dy,Br,KO, F} for these connections and the monographs \cite{D, Me} for the  background
on random matrices.
These Bessel processes are classified via root systems and by coupling or multiplicity parameters $k$
which govern the interactions; see  \cite{CDGRVY, R1, RV1, DF, DV} and references therein for the details.
Moreover, similar  systems on  $[-1,1]$ can be described via  Jacobi processes
on alcoves in $\mathbb R^N$ which have the distributions of $\beta$-Jacobi ensembles as invariant distributions; see \cite{Dem, RR, V2}.

Recently, several limit theorems were derived   when one or several 
 multiplicity parameters $k$ tend to infinity;
see  \cite{AKM1, AKM2, AV1, AV2, HV, V, VW}. In particular,  \cite{AV1, AV2, V, VW}
contain 
central limit theorems for  Bessel processes for $k\to\infty$, and  \cite{HV} contains a corresponding result for $\beta$-Jacobi ensembles.
In the most interesting cases,
the freezing limits are  $N$-dimensional centered Gaussian distributions 
where the inverses $\Sigma_N^{-1}$
of the covariance matrices
$\Sigma_N$ can be computed explicitly in terms of the zeros of classical
orthogonal polynomial of order $N$.
In particular,
for the Bessel processes with the root systems $A_{N-1}$ and $B_N$,
these orthogonal polynomials are classical Hermite and Laguerre polynomials, and the associated freezing WLTs
for the Bessel processes with start in the origin are closely related with corresponding 
WLTs of Dumitriu and  Edelman \cite{DE2} for  $\beta$-Hermite and  $\beta$-Laguerre
ensembles respectively for $\beta\to\infty$.
However, the statements of these WLTs in \cite{DE2}  and \cite{V, AV1} are slightly different, as in 
\cite{DE2} explicit formulas for the covariance matrices $\Sigma_N$
are given instead of the inverse matrices $\Sigma_N^{-1}$ in \cite{V, AV1}.
Both types of formulas involve the zeros of the $N$-th Hermite or Laguerre polynomial (with a suitable parameter)
respectively, but it seems to be 
difficult  to verify that the approaches in \cite{DE2}  and \cite{V, AV1} are equivalent.
In fact, for small dimensions $N$, this equivalence was verified numerically.
The different formulas in \cite{DE2}  and \cite{V, AV1} were one of the starting points for this paper.
In fact, by \cite{AV1}, the matrices $\Sigma_N^{-1}$ and thus the $\Sigma_N$ can be diagonalized: 
$\Sigma_N^{-1}$ has the eigenvalues $1,2,\ldots,N$ in the $A_{N-1}$-case and $2,4,\ldots,2N$ in the $B_N$-case,
and in both cases the transformation matrices can be described in terms of a finite sequence $(Q_k)_{k=0,\ldots,N-1}$
of orthogonal polynomials which are orthogonal w.r.t.~the empirical
measure of the zeros of the $N$-th associated Hermite or Laguerre polynomial respctively. We show
that this diagonalization of  $\Sigma_N^{-1}$ also leads to an explicit three-term-recurrence relation for the sequence
$(Q_k)_{k=0,\ldots,N-1}$; see the end of Section 2 in the $\beta$-Hermite case.
This recurrence immediately shows that the sequences  $(Q_k)_{k=0,\ldots,N-1}$ are dual in the sense of de Boor and Saff
\cite{BS}
to the finite parts
$(H_k)_{k=0,\ldots,N-1}$ and $(L_k^{(\alpha)})_{k=0,\ldots,N-1}$ of the associated Hermite and Laguerre polynomials respectivey.
With this knowledge in mind we reprove this fact in a more elegant way in Section 4 via this duality theory;
see also \cite{VZ, I} for details. It turns out that this approach also works for the freezing limits
of the $\beta$-Jacobi ensembles in \cite{HV}  where Jacobi polynomials and their zeros appear in a similar way as for the
$\beta$-Hermite and $\beta$-Laguerre ensembles.

After having identified the polynomials  $(Q_k)_{k=0,\ldots,N-1}$ as dual polynomials in all these 3 classical matrix ensembles,
we determine new explicit formulas for the entries of  $\Sigma_N$ in Section 4. It turns out that our
approach to 
$\Sigma_N$ for the $\beta$-Hermite and $\beta$-Laguerre ensembles in the freezing limit leads to formulas
 different from \cite{DE2}, and  we are not able to
 check equality of these formulas for arbitrary dimensions $N$.
 On the other hand, our formulas in the $\beta$-Hermite limit were derived recently by completely different methods by Gorin and  Kleptsyn \cite{GK}; see also \cite{GoM} for related results. 
 The equality of $\Sigma_N$ here (and in  \cite{GK} in the $\beta$-Hermite case) and in \cite{DE2} lead to
  some unknown  connections between the
  zeros of the $N$-th Hermite or Laguerre polynomial and the corresponding polynomials of order $0,1,2,\ldots,N-1$;
  see Corollary \ref{equality-of-variances-h} below for the details in the Hermite case.
We  point out that in the  $\beta$-Hermite limit case, our 
formulas for the entries of  $\Sigma_N$ have the same complexity as those in \cite{DE2}, while in the Laguerre case our
formulas have the same form as in the  $\beta$-Hermite case while  the formulas in  \cite{DE2} are
considerably more complicated. Moreover, in the  $\beta$-Jacobi case, our formulas  for  $\Sigma_N$ also have the same structure
while corresponding results based on  tridiagonal random matrix models as in \cite{K, KN}
seem to be unknown.

In the remaining sections we use our formulas for  $\Sigma_N=(\sigma_{i,j})_{i,j=1,\ldots,N}$ in order to derive limit results
for $\sigma_{N,N}$ for $N\to\infty$ in the $\beta$-Hermite and $\beta$-Laguerre case which involves the Airy function 
$\mathsf {Ai}$ and the $r$ largest zeros $a_r<a_{r-1}<\ldots< a_1<0$ of $\mathsf {Ai}$. For a discussion of $\mathsf {Ai}$ we refer to 
\cite{NIST, VS}.
In particular, for the largest eigenvalues in the $\beta$-Hermite case we obtain the following theorem which summarizes the main
results of  Section 5. For the precise definition of the  Bessel processes 
$(X_{t,k}^N)_{t\geq 0}$ of type $A_{N-1}$ we refer to the beginning of Section 2 below.

\begin{theorem}\label{localconvfnrth-introduction}
  Let $r\in\mathbb{N}$. For $N\ge r$ 
consider  the  Bessel processes 
$$(X_{t,k}^N)_{t\geq 0}=(X_{t,k,1}^N,\ldots,X_{t,k,N}^N )_{t\ge0}$$ of type $A_{N-1}$  with start in $0\in \mathbb R^N$.
Then, for each $t>0$, 
\begin{equation}\label{CLT-max-limit-main-Hermite-intro}
        \lim_{N\to\infty}\left(\lim_{k\to\infty}N^\frac{1}{6}\sqrt{2k}\left(\frac{X_{t,k,N-r+1}^N}{\sqrt{2kt}}-z_{N-r+1,N}\right)\right)=
G_r
        \end{equation}
in distribution with some $\mathcal{N}(0,\sigma_{max, r}^2)$-distributed random variable $G_r$  with variance 
	\begin{align*}
	\sigma_{max, r}^2=
	\int_{0}^{\infty}\frac{\mathsf {Ai}(x+a_r)^2}{\mathsf {Ai}'(a_r)^2x}dx =
	\begin{cases}
	0.582\ldots\text{for }r=2\\
		0.472\ldots\text{for }r=3\\
			0.407\ldots\text{for }r=4\\
			\ldots.
	\end{cases}
	\end{align*}
where $z_{N-r+1,N}$ is the $r$-th largest zero of the classical Hermite polynomial $H_N$, which satisfies by some classical formula of Plancherel-Rotach 
(see e.g.~\cite{T})
\begin{equation}\label{Planch-rot-r-intro}
\frac{z_{N-r+1,N}}{\sqrt{2N}}=1-\frac{|a_r|}{2N^\frac{2}{3}}+O(N^{-1}) \quad\quad(N\to\infty).
\end{equation}
Moreover, the variances $\sigma_{max, r}^2$ tend to $0$ for $r\to\infty$.
\end{theorem}

For $r=1$, this result  was stated by Dumitriu and Edelman (Corollary 3.4 in \cite{DE2}), where
the result there contains a misprint and the proof is sketched only. Moreover, as the proof in  \cite{DE2} uses
a different formula for $\sigma_{N,N}$, they obtain 
\begin{equation}\label{variance-DE-intro}
\sigma_{max, 1}^2
        =2\frac{\int_{0}^{\infty}\mathsf {Ai}^4(x+a_1)dx}{\left(\int_{0}^{\infty}\mathsf
{Ai}^2(x+a_1)dx\right)^2}=2\int_{0}^{\infty}\left(\frac{\mathsf
{Ai}(x+a_1)}{\mathsf {Ai}'(a_1)}\right)^4dx.
        \end{equation}
A numerical computation shows that the value of (\ref{variance-DE-intro}) seems to be equal to that in Theorem
\ref{localconvfnrth-introduction} for $r=1$.
Unfortunately, we are not able to verify this equality in an analytic way, as our suggested identity does not seem to fit to known
identities for the Airy function  as, for example, in \cite{VS}. 
Besides this result for the largest eigenvalues in the $\beta$-Hermite case we  also derive a corresponding result for the largest
eigenvalues of the frozen Laguerre ensembles by the same methods.
Our approach also leads to a corresponding result for the smallest eigenvalues of frozen Laguerre ensembles, namely, at the hard edge (see \cite{A} for a derivation), and we expect that it will apply to the extremal eigenvalues of frozen Jacobi ensembles where then  Bessel functions instead of the Airy function appear.

This paper is organized as follows: In Section 2 we recapitulate some facts on Bessel processes
of type $A_{N-1}$ and $\beta$-Hermite ensembles. In particular the WLTs in the freezing limit from \cite{DE2, V, AV2}
and the covariance matrices   $\Sigma_N$ and their inverses are discussed there. Moreover we derive the 
 three-term-recurrence relation for 
 $(Q_k)_{k=0,\ldots,N-1}$ there via matrix analysis.
Section 3 is then devoted to the corresponding known results for the   $\beta$-Laguerre and  $\beta$-Jacobi ensembles
from \cite{DE2, V, AV2, HV}.
Then, in Section 4 we discuss general dual finite orthogonal polynomials and apply this to the classical polynomials.
In this way we obtain new formulas for the covariance matrices $\Sigma_N$ for all 3 classical types of
ensembles in a unifying way. These results are applied in Section 5 for the Hermite cases,
in order to determine of some entries of
$\Sigma_N$ for $N\to\infty$ in terms of Airy functions.
Finally, in Section 6, the corresponding limits in the Laguerre cases
are determined at the soft edge.

\section{WLTs for Hermite ensembles for $\beta\to\infty$}

In this section we  recapitulate some  WLTs for  the  root systems $A_{N-1}$ for
fixed $N\ge2$ and
 $\beta\to\infty$ from \cite{DE2,  AV2, V} where we add a new result in the end.
Here we have a one-dimensional coupling constant $\beta=2k\in[0,\infty[$ where the notation $k$ is usually used in the Bessel process community and  $\beta$ in  the random matrix community.
 The associated Bessel processes $(X_{t,k})_{t\ge0}$ are Markov processes
 on the closed Weyl chamber
$$C_N^A:=\{x\in \mathbb R^N: \quad x_1\le x_2\le\ldots\le x_N\}$$
 where the generator of the transition semigroup is 
\begin{equation}\label{def-L-A} L_Af:= \frac{1}{2} \Delta f +
 k \sum_{i=1}^N\Bigl( \sum_{j\ne i} \frac{1}{x_i-x_j}\Bigr) \frac{\partial}{\partial
x_i}f ,
 \end{equation}
and we assume reflecting boundaries.  The transition probabilities of these
processes for $t>0$
can be expressed in terms of multivariate Bessel functions of type  $A_{N-1}$; see  
\cite{R1, RV1}.
Here we only recapitulate  that under the starting condition $X_{0,k}=0\in C_N^A $,
the random variable  $X_{t,k}$ has the  Lebesgue-density
\begin{equation}\label{density-A-0}
 \frac{c_k}{t^{\gamma_A+N/2}} e^{-\|y\|^2/(2t)} \cdot \prod_{i<j}(y_j-y_i)^{2k}\> dy
\end{equation}
on $C_N^A$ for $t>0$ with the constants
$$\gamma_A:= kN(N-1)/2,\quad\quad  c_k^A:=\frac{N!}{(2\pi)^{N/2}}
\cdot\prod_{j=1}^{N}\frac{\Gamma(1+k)}{\Gamma(1+jk)}.$$
 Up to scaling, this is simply the distribution of the ordered spectra
of the  $\beta$-Hermite ensembles of Dumitriu and Edelman
\cite{DE1}. Using this  tridiagonal $\beta$-Hermite model, Dumitriu and Edelman
(Theorem 3.1 of \cite{DE2}) 
 derived the following WLT \ref{clt-main-a-DE} for $\beta=k\to\infty$ 
where the data of the limits are given in terms of the ordered zeros
 $z_{1,N}<\ldots< z_{N,N}$ of the $N$-th Hermite polynomial $H_N$. 
For this we recall that, as usual (see e.g.~\cite{S}),
the Hermite  polynomials $(H_n)_{n\ge 0}$ are orthogonal w.r.t.
 the density  $e^{-x^2}$ with the three-term-recurrence relation
\begin{equation}\label{hermite-classical}
H_0=1, \quad H_1(x)=x, \quad H_{n+1}(x)=2x H_n(x)-2nH_{n-1}(x) \quad(n\ge1).
\end{equation}
The  Hermite polynomials, orthonormalized w.r.t.~the probability measure
$\pi^{-1/2}e^{-x^2}$, will be denoted by
$(\tilde H_n)_{n\ge0}$. By (5.5.1) of \cite{S}, we thus have
\begin{equation}\label{hermite-orthonormal}
 \tilde H_n(x)=\frac{1}{2^{n/2}\sqrt{n!}} H_n(x) \quad\quad(n\ge0).
\end{equation}

\begin{theorem}\label{clt-main-a-DE}
Consider the  Bessel processes $(X_{t,k})_{t\ge0}$ of type $A_{N-1}$  with start in
$0\in C_N^A$.
Then, for each $t>0$,
$$\frac{X_{t,k}}{\sqrt t} -  \sqrt{2k}\cdot (z_{1,N},\ldots,z_{N,N})$$
converges for $k\to\infty$ to the centered $N$-dimensional normal distribution
$N(0,\Sigma_N)$
with the  covariance matrix $\Sigma_N=(\sigma_{i,j}^2)_{i,j=1,\ldots,N}$ given by
\begin{equation}\label{covariance-a-de}
\sigma_{i,j}^2= \frac{ \sum_{l=0}^{N-1} \tilde H_l^2(z_{i,N}) \tilde H_l^2(z_{j,N})
+ \sum_{l=0}^{N-2}\tilde H_{l+1}(z_{i,N}) \tilde H_l(z_{i,N}) \tilde
H_{l+1}(z_{j,N})\tilde H_{l}(z_{j,N})}{
\sum_{l=0}^{N-1}\tilde H_l^2(z_{i,N}) \cdot \sum_{l=0}^{N-1}\tilde H_l^2(z_{j,N}) }.
\end{equation}
\end{theorem}

This WLT was proved in \cite{V} by  a different method which leads to
an explicit formula for the inverse matrix $\Sigma_N^{-1}$, but not for $\Sigma_N$.
This approach  was improved in \cite{AV2}
from the starting point $0\in C_N^A$ to arbitrary starting points $x\in C_N^A$ where
 this WLT is slightly complicated for  $x\ne0$
 as the root system $A_{N-1}$ is not reduced on $\mathbb R^N$ for $N\ge2$.
 This means that with the vector ${\bf 1}:=(1,\ldots,1)\in \mathbb R^N$,  
the space $\mathbb R^N$ can be decomposed into   $\mathbb R\cdot {\bf 1}$
and its orthogonal complement
 $$ {\bf 1}^\perp=\{x\in\mathbb R^N:\> \sum_{i=1}^N x_i=0\} \subset \mathbb R^N$$ 
so that the associated Weyl group, namely the symmetric group $\mathfrak{S}_N$ here, acts on
both spaces separately. 
We denote the orthogonal projections from  $ \mathbb R^N$ onto  $\mathbb R\cdot {\bf
1}$ and  $ {\bf 1}^\perp$ by
 $\pi_{\bf 1}$ and $\pi_{{\bf 1}^\perp}$ respectively. In particular, for  $x\in
\mathbb R^N$, we have
 $\pi_{\bf 1}(x)= \bar x{\bf 1} $ for the center of gravity
 $\bar x=\frac{1}{N}\sum_{i=1}^N x_i$ of the particles. With these notations, the
following WLT is shown in \cite{AV2}:

\begin{theorem}\label{clt-main-a-general-x}
Consider the  Bessel processes $(X_{t,k})_{t\ge0}$ of type $A_{N-1}$ on $C_N^A$ 
with an arbitrary fixed
 starting point $ x\in C_N^A$.
Then, for each $t>0$,
$$\frac{X_{t,k}}{\sqrt t} -  \sqrt{2k}\cdot (z_{1,N},\ldots,z_{N,N})$$
converges for $k\to\infty$ in distribution to the  $N$-dimensional normal distribution
 $N(\pi_{\bf 1}(x/\sqrt{t}),\Sigma_N)$
where  the inverse  $\Sigma_N^{-1}=:S_N=(s_{i,j})_{i,j=1}^N$ of the covariance
matrix $\Sigma_N$ is given by
\begin{equation}\label{covariance-matrix-A}
s_{i,j}:=\left\{ \begin{array}{r@{\quad\quad}l}  1+\sum_{l\ne i}
(z_{i,N}-z_{l,N})^{-2} & \text{for}\quad i=j \\
   -(z_{i,N}-z_{j,N})^{-2} & \text{for}\quad i\ne j.  \end{array}  \right.
\end{equation}
\end{theorem}

In \cite{AV2}, the eigenvalues and eigenvectors of $S_N$  were determined via 
 finite  orthogonal polynomials which are orthogonal w.r.t.~the empirical
measures
\begin{equation}\label{empirical-measure-a}
\mu_N:=\frac{1}{N}(\delta_{z_{1,N}}+\ldots+\delta_{z_{N,N}})\in M^1(\mathbb R)
\end{equation}
of the zeros of $H_N$. For the general theory of (finite)  orthogonal polynomials 
we refer to the monographs \cite{C, I}. In fact, 
Gram-Schmidt orthonormalization of the monomials $x^n$, $n=0,\ldots,N-1$,  leads to
a unique 
 finite sequence of orthogonal polynomials 
$\{Q_n^{(N)}\}_{n=0}^{N-1}$ with positive leading coefficients, $\text{deg}
[Q_n^{(N)}]=n$, and with
\begin{equation}\label{normalization-ops-a}
\sum_{i=1}^N Q_n^{(N)}(z_{i,N}) Q_m^{(N)}(z_{i,N})=\delta_{n,m}
\quad\quad(n,m=0,\ldots, N-1).
\end{equation}
We then have the following result  by Theorem 3.1 of \cite{AV2}:

\begin{theorem}\label{ev-a}
For each $N\ge2$, the matrix  $S_N$ in Theorem \ref{clt-main-a-general-x} has the
eigenvalues $\lambda_k=k$ for
$k=1,2, \ldots, N$.
Moreover, for  $n=1,\ldots,N$, the vector
$$\bigl(Q_{n-1}^{(N)}(z_{1,N}), \ldots, Q_{n-1}^{(N)}(z_{N,N})\bigr)^T$$
is an eigenvector of $S_N$  for the eigenvalue $n$.
\end{theorem}

The finite orthogonal polynomials $\{Q_n^{(N)}\}_{n=0}^{N-1}$ admit a three-term-recurrence relation which can be derived from the proof of Theorem \ref{ev-a} in \cite{AV2}. This explicit relation will be essential below.  It will be convenient to write down this  relation for the
monic orthogonal polynomials
$\{\hat Q_n^{(N)}\}_{n=0}^{N-1}$ associated with $\{Q_n^{(N)}\}_{n=0}^{N-1}$, i.e.,  $Q_k^{(N)}=l_k\hat Q_k^{(N)}$ with the leading coefficients $l_k>0$  of
 $Q_k^{(N)}$:

\begin{proposition}\label{three-term-dual-Hermite}
The  monic  orthogonal polynomials $\{\hat Q_n^{(N)}\}_{n=0}^{N-1}$ satisfy
        \begin{equation}\label{3termqhermite}
\hat{Q}_{0}^{(N)}=1, \> \hat{Q}_{1}^{(N)}(x)=x,\>
        \hat{Q}_{k+1}^{(N)}(x)=x\hat{Q}_{k}^{(N)}(x)-\left(\frac{N-k}{2}\right)\hat{Q}_{k-1}^{(N)}(x)
        \end{equation}
for $k=1,...,N-2$.
\end{proposition}

\begin{proof} For $k=1,\ldots,N$ consider the vector
$v_k=(z_{1,N}^{k-1},...,z_{N,N}^{k-1})^T$  as 
 in the proof of Theorem 2 in \cite{AV2}. Eq.~(3.5)  in \cite{AV2} shows
that  the $i$-th component of the vector $(S_N-kI_N)v_{k}$ has the form
$$\big((S_N-kI_N)v_{k}\big)_i=-\left(N-\frac{k-1}{2}\right)(k-2)z_{i,N}^{k-3}+s_k(z_{i,N})$$
with some polynomial $s_k$ of order at most $k-5$.        
Therefore, if we put
\begin{align}\label{defek}
        e_k:=-\left(N-\frac{k-1}{2}\right)\frac{k-2}{2},
        \end{align}
we obtain
\begin{align}\label{eqforek}
        \big((S_N-kI_N)&(v_{k}+e_kv_{k-2})\big)_i\notag\\
        =&\big((S_N-kI_N)(v_{k})+(S_N-(k-2)I_N)(e_kv_{k-2})-2e_kv_{k-2}\big)_i\notag\\
        =&-\left(N-\frac{k-1}{2}\right)(k-2)z_{i,N}^{k-3}
        -2e_kz_{i,N}^{k-3}+r_k(z_{i,N})\notag\\
=&r_k(z_{i,N})
        \end{align}
with some polynomial $r_k$ of degree at most $k-5$.
On the other hand, by  the proof of Theorem 2 in \cite{AV2}, there 
exist polynomials $p_{k}$ of order at most $k-5$  with
        \begin{align}\label{eqforrk}
        (S_N-kI_N)(p_{k}(z_{1,N}),...,p_{k}(z_{N,N}))=(r_{k}(z_{1,N}),...,r_{k}(z_{N,N})).
        \end{align}
(\ref{eqforek}) and (\ref{eqforrk}) imply that
        \begin{align*}
        \left(S_N-kI_N\right)&\left(z_{i,N}^{k-1}+e_kz_{i,N}^{k-3}-p_{k}(z_{i,k})\right)_{i=1,...,N}=(0)_{i=1,\ldots,N}.
        \end{align*}
In summary, we find monic polynomials $(q_k)_{k=0,\ldots,N-1}$ with
$\text{deg}\> q_k=k$ and $q_k(z)= z^{k}+e_{k+1}z^{k-2}-p_{k+1}(z)$ such that the vector
 $(q_k(z_{1,N}),\ldots,q_k(z_{N,N}))^{T}$ is an eigenvector of the matrix $S_N$ with the eigenvalue $k+1$.
Because the eigenvectors of $S_N$ are orthogonal, we conclude that 
 $(q_k)_{k=0,\ldots,N-1}$ is equal to the finite monic sequence 
        $\{\hat{Q}_{k}^{(N)}\}_{k=0,...,N-1}$ of orthogonal polynomials w.r.t.~the 
 measure $\mu_N$. As the measure  $\mu_N$ is symmetric, the $\hat{Q}_{k}^{(N)}$ have
a three-term recurrence
of the form
        \begin{align*}
        \hat{Q}_{k}^{(N)}(x)=x\hat{Q}_{k-1}^{(N)}(x)-b_k \hat{Q}_{k-2}^{(N)}(x)
        \end{align*}
with some coefficients $b_k>0$. This leads to
        \begin{align}\label{recrelation}
        x^k+e_{k+1}x^{k-2}=x^k+e_kx^{k-2}-b_k x^{k-2}+\text{terms of lower degree }.
        \end{align}
        Hence, by (\ref{defek}) and (\ref{recrelation}) for $k=1,...,N-1$,
        \begin{align*}
        b_k&=e_{k}-e_{k+1}=\left(N-\frac{k}{2}\right)\left(\frac{k-1}{2}\right)-\left(N-\frac{k-1}{2}\right)\left(\frac{k-2}{2}\right)\\
        &=\left(N-\frac{k}{2}\right)\left(\frac{k-1-(k-2)}{2}\right)-\frac{1}{2}\left(\frac{k-2}{2}\right)\\
        &=\frac{1}{2}\left(\frac{2N-k-k+2}{2}\right)=\frac{N-k+1}{2}.
        \end{align*}
        This leads to the three-term-recursion in the statement.
\end{proof}

The three-term-recurrence (\ref{hermite-orthonormal}) of the Hermite polynomials
implies 
that $H_n$ has the leading coefficient $2^n$. Hence,  by 
(\ref{hermite-orthonormal}), the monic 
 Hermite polynomials $(\hat{H}_n :=2^{-n}H_n)_{n\ge0}$ satisfy the 
three-term-recurrence 
\begin{equation}\label{monic-hermite-3-term}
H_0=1, \quad H_1(x)=x, x\hat{H}_n(x)=\hat{H}_{n+1}(x)+\frac{n}{2}\hat{H}_{n-1}(x)
\quad(n\ge1).
\end{equation}
This recurrence is  related to that in Proposition \ref{three-term-dual-Hermite} via the theory of dual orthogonal polynomials
by de Boor and Saff \cite{BS}. We  show in Section 4 that this connection between the sequences  $(\hat{H}_n)_{n\ge0}$ and $(\hat{Q}_{k}^{(N)})_{k=0,\ldots,N-1}$
also holds  for further random matrix models and the associated orthogonal polynomials. In this way, Proposition \ref{three-term-dual-Hermite}
can be also proved via the theory of dual orthogonal polynomials.

\section{ WLTs for Laguerre and Jacobi ensembles as $\beta\to\infty$}

In this section we  recapitulate some  WLTs for  $\beta\to\infty$ from \cite{DE2, AV2, V, HV}
for the Bessel processes  of type $B_N$ and the Jacobi case.

We first turn to  Bessel processes $(X_{t,k})_{t\ge0}$ of type $B_N$.
 These Markov processes live in the closed Weyl chamber
$$C_N^B:=\{x\in {\b R}^N:0\leq x_1\leq x_2\leq\cdots\leq x_N\},$$
and the generator of their transition semigroup is
\begin{equation}
L_Bf:=\frac{1}{2} \Delta f + k_1 \sum_{i=1}^{N}\frac{1}{x_i}\frac{\partial}{\partial
x_i}f
+k_2 \sum_{i=1}^N\Bigl( \sum_{j\ne i} \frac{1}{x_i-x_j}+\frac{1}{x_i+x_j}\Bigr)
\frac{\partial}{\partial x_i}f
\end{equation}
with  multiplicities $k_1,k_2\ge0$ and reflecting boundaries.
We write the multiplicities as $(k_1,k_2)=(\kappa\cdot\nu,\kappa)$ with $\nu,\kappa\ge0$ where the parameter $\beta$ from random matrix theory is
$\beta=2\kappa$.
The  transition probabilities of these processes for $t>0$ are again known; see 
\cite{R1, RV1}.
In particular, under the starting condition $X_{0,k}=0\in C_N^A $,
 $X_{t,k}$ has the  Lebesgue-density
\begin{equation}\label{density-B-0}
 \frac{c_k}{t^{\gamma_B+N/2}} e^{-\|y\|^2/(2t)} \cdot
\prod_{i<j}(y_j^2-y_i^2)^{2k_2}\cdot \prod_{i=1}^N y_i^{2k_1}\> dy
\end{equation}
on $C_N^B$ for $t>0$ with some known normalizations $c_k^B>0$ and
$$ \gamma_B(k_1,k_2)=k_2N(N-1)+k_1N.$$
 Up to scaling, these  distributions belong to the ordered spectra
of the  $\beta$-Laguerre ensembles in
\cite{DE1}. Using their  tridiagonal $\beta$-Laguerre models, Dumitriu and Edelman
\cite{DE2}
 derived a WLT for $\beta\to\infty$ 
where  the limits are given in terms of 
 the ordered zeros $z_{1,N}^{(\nu-1)}\leq\cdots  \leq z_{N,N}^{(\nu-1)}$ 
of the Laguerre polynomial $L_N^{(\nu-1)}$. 
We recapitulate that for $\alpha>-1$ the Laguerre polynomials
$(L_n^{(\alpha)})_{n\ge0}$
are orthogonal w.r.t. the density
 $\mathrm{e}^{-x}x^{\alpha}$ on $]0,\infty[$ as in  \cite{S} with the  three-term
recurrence relation
\begin{align}\label{laguerre-classical}
&L_0^{(\alpha)}=1, \quad L_1^{(\alpha)}(x)=-x+\alpha+1, \quad\notag\\
&(n+1) L_{n+1}^{(\alpha)}(x)=(-x+2n+\alpha+1)
L_n^{(\alpha)}(x)-(n+\alpha)L_{n-1}^{(\alpha)}(x) \quad(n\ge1).
\end{align}
The  Laguerre polynomials orthonormalized w.r.t. $\frac{1}{\Gamma(\alpha+1)}\mathrm{e}^{-x}x^{\alpha}$, the Gamma distribution,
 will be denoted by
$(\tilde L_n^{(\alpha)})_{n\ge0}$. By (5.1.1) of \cite{S}, we thus have
\begin{equation}\label{Laguerre-orthonormal}
 \tilde L_n^{(\alpha)}(x)={n+\alpha\choose n}^{-1/2} L_n^{(\alpha)}(x)
\quad\quad(n\ge0).
\end{equation}

Using these notations,
  Dumitriu and Edelman \cite{DE2} proved  for $\nu>0$ fixed,
$X_{0,(\beta\cdot\nu/2,\beta/2)}=0\in C_N^B$, and 
 $\beta\to\infty$
that with the vector $r\in C_N^B$ given by
\begin{equation}\label{def-vector-r-b} 
(z_{1,N}^{(\nu-1)},\ldots, z_{N,N}^{(\nu-1)})=(r_{1}^2,\ldots,r_{N}^2),
\end{equation}
the random variable
$$\frac{X_{t,(\beta\cdot\nu,\beta)}}{\sqrt t} -  \sqrt{\beta}\cdot r$$
converges in distribution to a centered normal random variable $N(0,\Sigma_N)$ with explicit formulas for the entries of $\Sigma_N$. As these formulas are quite complicated we skip them here.
Similar to the Hermite case, this WLT was extended  in \cite{V} to
arbitrary starting points as follows
 with explicit formulas for $\Sigma_N^{-1}$:

\begin{theorem}\label{clt-main-b}
Consider the  Bessel processes $(X_{t,k})_{t\ge0}$ of type $B_{N}$ on $C_N^B$ for
$k=(k_1,k_2)=(\kappa\cdot\nu,\kappa)$
 and $\kappa,\nu>0$ with start in $x\in C_N^B$.
Then, for each $t>0$,
$$\frac{X_{t,(\kappa\cdot\nu,\kappa)}}{\sqrt t} -  \sqrt{2\kappa}\cdot r$$
converges for $\kappa\to\infty$ to the centered $N$-dimensional distribution
$N(0,\Sigma_N)$
with the regular covariance matrix $\Sigma_N$ where
$\Sigma_N^{-1}=:S_N=(s_{i,j})_{i,j=1,\ldots,N}$ is given by
\begin{equation}\label{covariance-matrix-B}
s_{i,j}:=\left\{ \begin{array}{r@{\quad\quad}l}  1+\frac{\nu}{r_i^2}+\sum_{l\ne i}
(r_i+r_l)^{-2}+\sum_{l\ne i} (r_i-r_l)^{-2} & \text{for}\quad i=j, \\
   (r_i+r_j)^{-2}-(r_i-r_j)^{-2} & \text{for}\quad i\ne j.  \end{array}  \right.
\end{equation}
\end{theorem}

Again,  the eigenvalues and eigenvectors of  $S_N$  were determined via finite
 orthogonal polynomials in \cite{AV2}.
 For this we introduce the measures
\begin{equation}\label{orthogonality-measure-b1}
\mu_{N,\nu}:=\frac{1}{N(N+\nu-1)}(z_{1,N}^{(\nu-1)}\delta_{z_{1,N}^{(\nu-1)}}+\ldots+z_{N,N}^{(\nu-1)}\delta_{z_{N,N}^{(\nu-1)}}).
\end{equation}
As
\begin{equation}\label{sum-zeroes-b}
\sum_{k=1}^N z_{k,N}^{(\nu-1)}= N(N+\nu-1) 
\end{equation}
by Appendix C of \cite{AKM2}, these measures are probability measures.
We now define the unique associated orthogonal polynomials
$(Q_{k}^{(N,\nu)})_{k=0,\ldots,N-1}$
 with $\deg\> Q_{k}^{(N,\nu)}= k$, positive leading coefficients, and with the
normalization
\begin{equation}\label{eq:Laguerrenormalization}
\sum_{i=1}^N z_{i,N}^{(\nu-1)}Q_{k}^{(N,\nu)}(z_{i,N}^{(\nu-1)})^2=1
\quad\quad(k=0,\ldots,N-1).
\end{equation}
With this  normalization, by \cite{AV2} the matrices
\begin{equation}\label{trafo-matrix-b}
T_N:= (r_i\cdot Q_k^{(N,\nu)}(r_i^2))_{i=1,\ldots,N, k=0,\ldots,N-1}
\end{equation}
are orthogonal, and moreover we have the following

\begin{theorem}\label{ev-b}
For $N\geq 2$, the matrix $S_N$ in Theorem~\ref{clt-main-b} has the eigenvalues $\lambda_k=2(k+1)$ with corresponding eigenvectors
\[(r_1Q_{k}^{(N,\nu)}(r_1^2),\ldots,r_N Q_{k}^{(N,\nu)}(r_N^2))^T,\ k=0,1,\ldots,N-1.\]
In particular,  
$S_N= T_N \cdot \textup{diag}(2,4,\ldots,2N) \cdot T_N^T.$
\end{theorem}

The three-term recurrence relations of the  polynomials 
$(Q_{k}^{(N,\nu)})_{k=0,\ldots,N-1}$
can  be determined in the same way as in the proof of Proposition \ref{three-term-dual-Hermite}. We skip this derivation here, as we shall present a more elegant proof of these relations in Section~4 via dual orthogonal polynomials.

We next turn to the $\beta$-Jacobi ensembles. 
In contrast with the Bessel processes on noncompact spaces, we do not study the associated Jacobi processes, but turn immediately to their invariant distributions. It turns out that by \cite{HV}, it is convenient for our considerations to study these invariant distributions in a trigonometric form, that is, after performing the coordinate transformation
$$(t_1,\ldots,t_N)\longrightarrow (\cos(2t_1),\ldots,
\cos(2t_N)).$$
Up to this transformation we follow \cite{F, K, KN, Me, HV} and consider for $k_1,k_2,k_3\ge0$ the trigonometric $\beta$-Jacobi  random matrix ensembles 
with the joint eigenvalue distributions $\tilde\mu_{(k_1,k_2,k_3)}$ given by the Lebesgue densities
\begin{equation}\label{density-joint-trig}
\tilde c_k\cdot \prod_{1\leq i< j \leq N}\left(\cos(2t_j)-\cos(2t_i)\right)^{k_3}
\prod_{i=1}^N\Bigl(\sin(t_i)^{k_1}\sin(2t_i)^{k_2}\Bigr)
\end{equation}
on the trigonometric alcoves
$$\tilde A:=\{t\in\R^N|\frac{\pi}{2}\geq t_1\geq...\geq t_N\geq 0 \}$$
with a suitable Selberg normalization $\tilde c_k>0$ for
$k=(k_1,k_2,k_3)\in[0,\infty[^3$; see the survey \cite{FW} for explicit formulas.
In \cite{HV}, a WLT was derived which corresponds to the preceding freezing WLTs for Bessel processes. We write
$$(k_1,k_2,k_3)=\kappa\cdot(a,b,1),$$
where $a\ge0$ $b>0$ are fixed and $\kappa$ tends to infinity. By \cite{HV}, the limit can be described via the ordered zeros
of the classical Jacobi polynomials 
$P_N^{(\alpha,\beta)}$ with parameters
$$\alpha:=a+b-1>-1, \quad\beta=b-1>-1.$$
Please notice that here, $\beta$ is a parameter different from the $\beta$ in  random matrix theory.
We recapitulate that the Jacobi polynomials  $(P_n^{(\alpha,\beta)})_{n\ge0}$  are orthogonal polynomials
 w.r.t.~the weights $(1-x)^\alpha(1+x)^\beta$   on $]-1,1[$;
see  \cite{S}. We denote their ordered zeros by  $z_1\le \ldots\le z_N$ 
where we the suppress   $\alpha,\beta>-1$.  We now use the vector 
$z:=(z_1, \ldots, z_N)\in A$. The following WLT is shown in \cite{HV}:

\begin{theorem}\label{theoremCLT-transformed}
Let $a\ge0$, $b>0$.
Let $\tilde X_\kappa$ be $\tilde A$-valued  random variables with the distributions 
$\tilde\mu_{\kappa\cdot(a,b,1)}$  for $\kappa>0$.
 Then, for $\kappa\rightarrow\infty$
 $$\sqrt{\kappa}(\tilde X_\kappa-\tilde z)   \quad\quad\text{with} \quad\quad
\tilde z:= (\frac{1}{2}\arccos z_1, \ldots,\frac{1}{2}\arccos z_N)\in\tilde A $$
converges in distribution to  $N(0,\tilde\Sigma_N)$ where the  inverse 
of the  covariance matrix $\tilde\Sigma_N$ is given by
 $\tilde\Sigma_N^{-1}=:\tilde S_N=(\tilde s_{i,j})_{i,j=1,...,N}$ with
\begin{align*}
\tilde s_{i,j}=
\begin{cases}4\sum_{ l\ne j}
\frac{1-z_j^2}{(z_j-z_l)^2}+2(a+b)\frac{1+z_j}{1-z_j}+2b\frac{1-z_j}{1+z_j}
&\textit{ for }i=j\\
\frac{-4\sqrt{(1-z_j^2)(1-z_i^2)}}{(z_i-z_j)^2}&\textit{ for }i\neq j
\end{cases}.
\end{align*}
\end{theorem}

The eigenvalues and eigenvectors of $\tilde S_N$ can be determined explicitly. 
For this we introduce  finite families of polynomials which are orthogonal w.r.t.~the measures
\begin{equation}\label{orthogonality-measure-finite-jacobi}
\mu_{N,\alpha,\beta}:=(1-z_1^2)\delta_{z_1}+\ldots+(1-z_N^2)\delta_{z_N}.
\end{equation}
We consider the associated finite  orthonormal polynomials
$(Q_{l}^{(\alpha,\beta,N)})_{l=0,\ldots,N-1}$
with positive leading coefficients and the normalization
\begin{align}\label{orthgonality-equation-finite-jacobi}
\sum_{i=1}^N Q_{l}^{(\alpha,\beta,N)}(z_i)
Q_{k}^{(\alpha,\beta,N)}(z_i)(1-z_i^2)=\delta_{l,k} \quad\quad(k,l=0,\ldots,N-1).
\end{align}
By \cite{HV} we then have:

\begin{theorem}\label{ev-main}
The matrix $\tilde S_N$  has the  eigenvalues
$\lambda_k=2k(2N+\alpha+\beta+1-k)>0$ ($k=1,\ldots,N$) with the eigenvectors
$$v_k:=\left(Q_{k-1}^{(\alpha,\beta,N)}(z_1)\sqrt{1-z_1^2},\ldots,Q_{k-1}^{(\alpha,\beta,N)}(z_N)\sqrt{1-z_N^2}\right)^T.
$$
In particular, with the orthogonal matrix $T_N:=(v_1,\ldots,v_N)$,
$$\tilde S_N= T_N   \cdot \operatorname{diag}(2(2N+\alpha+\beta+1-1),\ldots,
2N(2N+\alpha+\beta+1-N))\cdot T_N^{T}.$$ 
\end{theorem}

As mentioned previously, the three-term recurrence relations of the  polynomials $(Q_{k}^{(\alpha,\beta,N)})_{k=0,\ldots,N-1}$ can  be determined in the same way as in the proof of Proposition \ref{three-term-dual-Hermite}. 
A more elegant proof of these relations  via dual orthogonal polynomials will be
given in the following section.

\section{De Boor-Saff duality and the covariance matrices}

In this section we use the theory of dual orthogonal polynomials 
  of de Boor and Saff \cite{BS}  to analyze the covariance 
matrices $\Sigma_N$ of the WLTs in the three cases of the preceding two sections.
This is motivated by the observation that the finite monic orthogonal polynomials 
$(\hat Q_k^{(N)})_{k=0,\ldots,N-1}$ in Section 2  are  the 
dual polynomials of the Hermite polynomials $(\hat H_k)_{k\ge0}$ by Proposition
\ref{three-term-dual-Hermite}.

To explain this we first  review  this theory from \cite{VZ} and  Section
2.11 of \cite{I}.
Let  $(\hat{P}_n)_{n=0}^\infty$ be a sequence of  monic orthogonal polynomials
where the orthogonality measure is a probability measure $\mu$ on $\mathbb{R}$ which
admits all moments, i.e.,
\begin{equation}
\int_\mathbb{R}\hat{P}_i(x)\hat{P}_j(x)d\mu(x)=\xi_i\delta_{ij}
\quad(i,j=0,1,2,\ldots)
\end{equation}
with some constants $\xi_i>0$ ($i\ge0$). We also have a  three-term recurrence
relation
\begin{equation}\label{monic-generic-3-term}
\hat{P}_0=1,\> \hat{P}_1(x)=x-a_0,\> 
x\hat{P}_n(x)=\hat{P}_{n+1}(x)+a_n\hat{P}_{n}(x)+u_n\hat{P}_{n-1}(x) \quad(n\ge1)
\end{equation}
with coefficients $a_n\in\mathbb R$ and  $u_n> 0$. We also consider the associated 
orthonormal polynomials
 $(\tilde{P}_n:=\xi_n^{-1/2}\hat{P}_n)_{n=0}^{\infty}$ with 
$\int_\mathbb{R}\tilde{P}_i(x)\tilde{P}_j(x)d\mu(x)=\delta_{ij}$.
These polynomials then satisfy the three-term recurrence
\begin{equation}\label{normal-generic-3-term}
\tilde{P}_0=1,\> \tilde{P}_1(x)=b_1^{-1}(x-a_0),\>
x\tilde{P}_{n}(x)=b_{n+1}\tilde{P}_{n+1}(x)+a_{n}\tilde{P}_{n}(x)+b_{n}\tilde{P}_{n-1}(x)
\quad(n\ge1)
\end{equation}
with $b_{n}=u_n\sqrt{\xi_{n-1}/\xi_{n}}=\sqrt{\xi_{n}/\xi_{n-1}}$ for $n\ge1$.
In particular we have 
\begin{equation}\label{xi-bn-un} 
\xi_0=1, \> \xi_n=u_n u_{n-1}\cdots u_1 \quad \text{and}\quad b_n=\sqrt{u_n} \quad
(n\ge1).
\end{equation}
Now fix $N>0$ arbitrarily. Gaussian quadrature implies that the finite set of polynomials
$(\tilde{P}_n)_{n=0}^{N-1}$
 obeys the discrete orthogonality relation
\begin{equation}
\sum_{i=1}^{N}w_i\tilde{P}_m(z_{i,N})\tilde{P}_n(z_{i,N})=\delta_{mn},
\end{equation}
with the  $N$ ordered zeros $z_{1,N}<\ldots<z_{n,N}$ of  $\tilde{P}_N$
and  the Christoffel numbers
\begin{equation}
w_i:=\frac{1}{b_N\tilde{P}_{N-1}(z_{i,N})\tilde{P}_{N}^\prime(z_{i,N})}>0
\quad(i=1,\ldots,N)
\end{equation}
which satisfy the normalization  $\sum_{i=1}^Nw_i=1$.

\begin{definition}\label{def-duality}
Let  $N>0$.  The monic polynomials $(\hat{Q}_{k,N})_{k=0}^{N-1}$ 
are called dual (in the de Boor-Saff sense) to $(\hat{P}_{n}(x))_{n=0}^{N-1}$ if
they satisfy the three-term recurrence
        \begin{align}\label{3-term-dual-monic}
&\hat{Q}_{0,N}=1, \> \hat{Q}_{1,N}(x)=x-a_{N-1},\\
&        x\hat{Q}_{k,N}(x)=\hat{Q}_{k+1,N}(x)+a_{N-k-1}\hat{Q}_{k,N}(x)+u_{N-k}\hat{Q}_{k-1,N}(x)
\quad(k=1,\ldots,N-2).
\notag        \end{align}
\end{definition}

This definition and Proposition  \ref{three-term-dual-Hermite} imply 
 that the polynomials
$(\hat Q_k^{(N)})_{k=0,\ldots,N-1}$ from Section 2 are in fact dual to the  monic
Hermite poynomials $\hat H_n$.

We now  recapitulate some consequences of this duality from  \cite{VZ}:

\begin{lemma} The dual monic polynomials $(\hat{Q}_{k,N})_{k=0}^{N-1}$ are
orthogonal w.r.t.~the discrete measure
$$\sum_{i=1}^Nw_i^*\delta_{z_{i,N}}$$
with the dual Christoffel numbers
\begin{align}\label{dual-Christoffel}
w_i^*=\frac{\tilde{P}_{N-1}(z_{i,N})}{b_N\tilde{P}_{N}^\prime(z_{i,N})}>0\quad
(i=1,\ldots,N)
\end{align}
which again satisfy $\sum_{i=1}^Nw_i^*=1$.
\end{lemma}

In particular, by (\ref{xi-bn-un}), the normalized dual polynomials
$(\tilde{Q}_{k,N})_{k=0}^{N-1}$
with
\begin{equation}\label{dual-orthogonality}
\sum_{i=1}^Nw_i^*\tilde{Q}_{m,N}(z_{i,N})\tilde{Q}_{n,N}(z_{i,N})=\delta_{mn}
\quad(m,n=0,\ldots,N-1)
\end{equation}
satisfy
\begin{align}\label{dual-normalized}
\tilde{Q}_{k,N}(x)=\frac{\hat{Q}_{k,N}}{b_{N}^2b_{N-1}^2\cdots b_{N-k}^2}.
\end{align}
In summary we obtain from \eqref{dual-normalized} and the three-term-recurrence in
Definition~\ref{def-duality}:

\begin{lemma}\label{orthonormal-generic-3-term}
        The orthonormal  dual polynomials $(\tilde{Q}_{k,N})_{k=0}^{N-1}$ satisfy the
three-term-recurrence relation
        \begin{align}\label{dual-recurrence-normal}
&\tilde{Q}_{0,N}=1,\> \tilde{Q}_{1,N}(x)=b_{N-1}^{-1}(x-a_{N-1}),\>\\
&        x\tilde{Q}_{k,N}(x)=b_{N-k-1}\tilde{Q}_{k+1,N}(x)+a_{N-k-1}\tilde{Q}_{k,N}(x)+b_{N-k}\tilde{Q}_{k-1,N}(x)\quad
( k\leq N-2).
        \notag\end{align}
\end{lemma}

\begin{remark}\label{remark-orthonormal-recurrence-upper-bound}
The monic three-term-recurrence (\ref{3-term-dual-monic}) is also available for
$k=N-1$, namely, we obtain a 
monic polynomial $\hat{Q}_{N,N}$. It can be easily seen (see \cite{VZ} or Section
2.11 of \cite{I}) that
$\hat{Q}_{N,N}= \hat P_N$ holds. Moreover, if we choose $b_0=0$ in
(\ref{dual-recurrence-normal}), then the recurrence
(\ref{orthonormal-generic-3-term}) remains valid for $k=N-1$, arbitrary polynomials
$\tilde{Q}_{k+1,N}$, and $x=z_{i,N}$ for $i=1,\ldots,N$.
\end{remark}

We next apply finite dual orthogonal polynomials in order to obtain additional information about 
the covariance matrices $\Sigma_N$ in the WLTs \ref{clt-main-a-general-x},
\ref{clt-main-b}, and 
\ref{theoremCLT-transformed}. In these cases, in $S_N=\Sigma_N^{-1}$, the Hermite
polynomials $H_n$,
 the Laguerre polynomials $L_n^{(\alpha)}$ with $\alpha=\nu-1$ and the Jacobi
polynomials $P_n^{(\alpha,\beta)}$
with $\alpha=a+b-1$, $\beta=b-1$ respectively appear. For fixed $N$ we now study the
associated orthonormal
dual polynomials which we denote by
 $({Q}_{k,N})_{k=0}^{N-1}$, $({Q}_{k,N}^{(\alpha)})_{k=0}^{N-1}$, and
$({Q}_{k,N}^{(\alpha,\beta)})_{k=0}^{N-1}$ respectively,
In all cases, let $z_{1,N}<\ldots<z_{N,N}$ be the ordered zeros of the $N$th
polynomial. With these notations we have:

\begin{lemma}\label{general-eigenvectors}
In the Hermite, Laguerre and Jacobi cases, orthonormal  eigenvectors of 
$S_N=\Sigma_N^{-1}$ are given  by the vectors
        \begin{align}
                \frac{1}{\sqrt{\kappa_N}}(\sqrt{\pi(z_{1,N})}\tilde{Q}_{j-1,N}(z_{1,N}),\ldots,\sqrt{\pi(z_{N,N})}\tilde{Q}_{j-1,N}(z_{N,N}))^T,\
1\leq j\leq N,
        \end{align}
        where the coefficients $\kappa_N$ and functions $\pi(x)$ are given by
        \begin{align}
                \pi(x)&=1,\quad\pi^{(\alpha)}(x)=x,\quad\pi^{(\alpha,\beta)}(x)=1-x^2,\text{ and }\\
                \kappa_N&=N,\quad\kappa_N^{(\alpha)}=N(N+\alpha),\quad
                \kappa_N^{(\alpha,\beta)}=\frac{4N(N+\alpha)(N+\beta)(N+\alpha+\beta)}{(2N+\alpha+\beta)^2(2N+\alpha+\beta-1)}\notag
        \end{align}
respectively.
\end{lemma}

\begin{proof}
We first consider the Hermite case. Here
$\hat{H}^\prime_{N}(x)=N\hat{H}_{N-1}(x)$ by Section 5.5 of
 \cite{S}. Hence, by
(\ref{dual-Christoffel}) and (\ref{dual-orthogonality}),
\begin{align}\label{normalization-ops-a-dual}
\sum_{i=1}^N\frac{1}{N}\tilde{Q}_{m,N}(z_{i,N})\tilde{Q}_{n,N}(z_{i,N})=\delta_{mn}.
\end{align}
If we compare this with the orthogonality (\ref{normalization-ops-a}) of the
polynomials $Q_n^{(N)}$ from Section~2,
 we conclude from Theorem \ref{ev-a} that the vectors
\begin{align}
\frac{1}{\sqrt{N}}(\tilde{Q}_{j-1,N}(z_{1,N}),\ldots,\tilde{Q}_{j-1,N}(z_{N,N}))^T
\end{align}
for $j=1,\ldots,N$ are orthonormal eigenvectors of $S_N$. 

We now turn to the Laguerre  case. By Section 4.6 of \cite{I} the monic Laguerre
polynomials satisfy 
\begin{align}
x\hat{L}^{(\alpha)\prime}_n(x)=n\hat{L}^{(\alpha)}_n(x)+n(n+\alpha)\hat{L}^{(\alpha)}_{n-1}(x).
\end{align}
In particular,
 $z_{i,N}^{(\alpha)}\hat{L}^{(\alpha)\prime}_N(z_{i,N}^{(\alpha)})=N(N+\alpha)\hat{L}^{(\alpha)}_{N-1}(z_{i,N}^{(\alpha)})$.
This, (\ref{dual-Christoffel}), and  (\ref{dual-orthogonality})
yield 
\begin{align}\label{normalization-ops-b-dual}
\sum_{i=1}^N\frac{z_{i,N}^{(\alpha)}}{N(N+\alpha)}\tilde{Q}_{m,N}^{(\alpha)}(z_{i,N}^{(\alpha)})\cdot \tilde{Q}_{n,N}^{(\alpha)}(z_{i,N}^{(\alpha)})=\delta_{mn}.
\end{align}
If we compare this with the orthogonality (\ref{eq:Laguerrenormalization}) of the
polynomials 
$Q_{k}^{(N,\nu)}$ with $\alpha=\nu-1$ in Section 3, we conclude from Theorem
\ref{ev-b} that the vectors
\begin{align}
(\sqrt{z_{1,N}^{(\alpha)}}&{Q}_{j-1,N}^{(\alpha)}(z_{1,N}^{(\alpha)}),\ldots,\sqrt{z_{N,N}^{(\alpha)}}{Q}_{j-1,N}^{(\alpha)}(z_{N,N}^{(\alpha)}))\notag\\
=&\frac{1}{\sqrt{N(N+\alpha)}}(\sqrt{z_{1,N}^{(\alpha)}}\tilde{Q}_{j-1,N}^{(\alpha)}(z_{1,N}^{(\alpha)}),\ldots,\sqrt{z_{N,N}^{(\alpha)}}\tilde{Q}_{j-1,N}^{(\alpha)}(z_{N,N}^{(\alpha)})).
\end{align}
for $j=1,\ldots,N$ are orthonormal eigenvectors of $S_N$.

Finally, the monic Jacobi polynomials $\hat R_N:= \hat{P}^{(\alpha,\beta)}_N$ satisfy
$$(1-(z_{i,N}^{(\alpha,\beta)})^2)\hat R_N^\prime(z_{i,N}^{(\alpha,\beta)})
=\frac{4N(N+\alpha)(N+\beta)(N+\alpha+\beta)}{(2N+\alpha+\beta)^2(2N+\alpha+\beta-1)}\hat R_N(z_{i,N}^{(\alpha,\beta)}).$$
This, (\ref{dual-Christoffel}),  and (\ref{dual-orthogonality}) show that
\begin{align}
\sum_{i=1}^N\frac{(1-(z_{i,N}^{(\alpha,\beta)})^2)(2N+\alpha+\beta)^2(2N+\alpha+\beta-1)}{4N(N+\alpha)(N+\beta)(N+\alpha+\beta)}&\tilde{Q}_{m,N}^{(\alpha,\beta)}(z_{i,N}^{(\alpha,\beta)})\tilde{Q}_{n,N}^{(\alpha,\beta)}(z_{i,N}^{(\alpha,\beta)})\notag\\
&=\delta_{mn}.
\end{align}
On the other hand, in Section 3 we imposed the following condition on ${Q}^{(\alpha,\beta)}_{m,N}(x)$:
\begin{align}
\sum_{i=1}^N(1-(z_{i,N}^{(\alpha,\beta)})^2){Q}_{m,N}^{(\alpha,\beta)}(z_{i,N}^{(\alpha,\beta)}){Q}_{n,N}^{(\alpha,\beta)}(z_{i,N}^{(\alpha,\beta)})=\delta_{mn}.
\end{align}
We thus conclude that the $i$-th component of the $j$-th eigenvector is equal to
\begin{align}
\sqrt{1-(z_{i,N}^{(\alpha,\beta)})^2}&{Q}^{(\alpha,\beta)}_{j-1,N}(z_{i,N}^{(\alpha,\beta)})\\
&=\sqrt{\frac{(1-(z_{i,N}^{(\alpha,\beta)})^2)(2N+\alpha+\beta)^2(2N+\alpha+\beta-1)}{4N(N+\alpha)(N+\beta)(N+\alpha+\beta)}}\tilde{Q}^{(\alpha,\beta)}_{j-1,N}(z_{i,N}^{(\alpha,\beta)}).\notag
\end{align}
This completes the proof.
\end{proof}

\begin{remark}
Notice that the comparison of the orthogonality  relations
(\ref{normalization-ops-a-dual}) and 
(\ref{normalization-ops-a}) in the Hermite case in the proof above yields that for all $k=0,\ldots,N-1$,
$\tilde{Q}_{k,N}(x)=\sqrt N \cdot {Q}_{k}^{(N)}(x)$.
This leads to a new  
proof of Proposition  \ref{three-term-dual-Hermite}.

In a similar way, the orthogonality  relations (\ref{normalization-ops-b-dual}) and
(\ref{eq:Laguerrenormalization}) yield
that for all $k=0,\ldots,N-1$, and $\alpha=\nu-1$, we have
$\tilde{Q}_{k,N}^{(\alpha)}= \sqrt{N(N+\alpha)}\cdot {Q}_{k}^{(N,\nu)}$.
 This leads to the three-term-recurrence for the polynomials $({Q}_{k}^{(N,\nu)})_{k=0,\ldots,N-1}$.
Moreover, a corresponding result is available in the Jacobi case.
\end{remark}

\begin{remark}
Notice that by the proof of Lemma \ref{general-eigenvectors}
in the Hermite, Laguerre, and Jacobi case the dual Christoffel numbers from
(\ref{dual-Christoffel}) have the form
\begin{align}
w_i^*=\frac{\hat{P}_{N-1}(z_{i,N})}{\hat{P}^\prime_{N}(z_{i,N})}=\frac{\pi(z_{i,N})}{\kappa_N}
\end{align}
with suitable constants $\kappa_N$ and polynomials $\pi$ of degrees 0,1, and 2
respectively. By \cite{VZ}, such simple relations for the dual Christoffel numbers are available  only
for the classical orthogonal polynomials. This also includes the Bessel polynomials which are limits of Jacobi polynomials;
see \cite[p. 124, (4.10.10) and (4.10.13)]{I}.
\end{remark}

In the next step we use the preceding results on dual orthogonal polynomials to compute the covariance matrices 
 $\Sigma_N$ from their inverses.
For this we write the 
recurrence \eqref{normal-generic-3-term} for general orthonormal polynomials $(\tilde{P}_{n})_{n\ge0}$ for $n\le N$  at the $N$ ordered
zeros $z_{i,N}$ of $\tilde{P}_{N}$ as
 the
eigenvalue equation
\begin{align}
        \left(\begin{array}{ccccc}
        a_0 & b_1 & & & \\
        b_1 & a_1 & b_2 & & \\
        & b_2 & \ddots & \ddots & \\
        & & \ddots & a_{N-2} & b_{N-1}\\
        & & & b_{N-1} & a_{N-1}
        \end{array}\right)
        \left(\begin{array}{c}
        \tilde{P}_{0}(z_{i,N})\\
        \tilde{P}_{1}(z_{i,N})\\
        \tilde{P}_{2}(z_{i,N})\\
        \vdots\\
        \tilde{P}_{N-1}(z_{i,N})
        \end{array}\right)=
        z_{i,N}\left(\begin{array}{c}
        \tilde{P}_{0}(z_{i,N})\\
        \tilde{P}_{1}(z_{i,N})\\
        \tilde{P}_{2}(z_{i,N})\\
        \vdots\\
        \tilde{P}_{N-1}(z_{i,N})
        \end{array}\right)
\notag\end{align}
of an $N\times
N$-dimensional matrix. The zeros $\{z_{i,N}\}_{i=1}^N$ are the eigenvalues of this symmetric matrix and are distinct; this yields that the eigenvectors of this matrix are orthogonal and unique up to a constant coefficient. On the other hand, Lemma~\ref{orthonormal-generic-3-term} and Remark~\ref{remark-orthonormal-recurrence-upper-bound} show that
\begin{align}
        (\tilde{Q}_{N-1,N}(z_{i,N}),\ldots,\tilde{Q}_{0,N}(z_{i,N}))^T
\end{align}
is also an eigenvector of this matrix for the eigenvalue $z_{i,N}$. It
follows that
\[\left(\begin{array}{c}
\tilde{P}_{0}(z_{i,N})\\
\tilde{P}_{1}(z_{i,N})\\
\tilde{P}_{2}(z_{i,N})\\
\vdots\\
\tilde{P}_{N-1}(z_{i,N})
\end{array}\right)=
c_{i,N}\left(\begin{array}{c}
\tilde{Q}_{N-1,N}(z_{i,N})\\
\tilde{Q}_{N-2,N}(z_{i,N})\\
\tilde{Q}_{N-3,N}(z_{i,N})\\
\vdots\\
\tilde{Q}_{0,N}(z_{i,N})
\end{array}\right),\]
with a constant $c_{i,N}\ne 0$. The last row of this equation and $\tilde{Q}_{0,N}(x)=1$ give
\begin{equation}\label{cifirstexpression}
c_{i,N}=\tilde{P}_{N-1}(z_{i,N}).
\end{equation}
We remark that $c_{i,N}$ usually has the sign $(-1)^{N-i}$. This follows  from the well-known intelacing property of the zeros of  $\tilde{P}_{N-1}(x)$ and $\tilde{P}_{N}(x)$ together with the assumption that the leading coefficient of   $\tilde{P}_{N-1}(x)$ is positive.
This assumption holds for the Hermite and Jacobi cases. The Laguerre case will be handled below.

The constants $c_{i,N}$ can be also determined from an eigenvalue equation for the inverse matrices  $S_N=\Sigma_N^{-1}$ 
for our random matrix ensembles. In fact, as the vectors in Lemma~\ref{general-eigenvectors} form an orthogonal
matrix in each of the cases considered there, we see that all rows and all columns of that matrix are orthogonal. Hence,
\[\frac{\sqrt{\pi(z_{i,N})\pi(z_{k,N})}}{\kappa_N
c_{i,N}c_{k,N}}\sum_{j=0}^{N-1}\tilde{P}_j(z_{i,N})\tilde{P}_j(z_{k,N})=\delta_{i,k}  \quad\text{for all}\quad 1\leq i,k\leq N.\]
In particular, for $i=k$,
\begin{equation}
c_{i,N}=\pm\sqrt{\frac{\pi(z_{i,N})}{\kappa_N}\sum_{j=0}^{N-1}\tilde{P}_j^2(z_{i,N})} \quad 1\leq i\le N.
\end{equation}
Using the sign of  $c_{i,N}$ above, we conclude that in the Hermite and Jacobi cases
\begin{equation}\label{cisecondexpression}
        c_{i,N}=(-1)^{N-i}\sqrt{\frac{\pi(z_{i,N})}{\kappa_N}\sum_{j=0}^{N-1}\tilde{P}_j^2(z_{i,N})}.
\end{equation}

In the Laguerre case,  the leading coefficient of  $L_{N-1}^{(\alpha)}(x)$  has the sign $(-1)^{N-1}$. In this case, we obtain
\begin{equation}\label{cilastexpression}
        c_{i,N}^{(\alpha)}=(-1)^{i-1}\sqrt{\frac{\pi^{(\alpha)}(z_{i,N}^{(\alpha)})}{\kappa_N^{(\alpha)}}\sum_{j=0}^{N-1}(\tilde{L}_j^{(\alpha)}(z_{i,N}))^2}.
\end{equation}
These observations now lead to the following representation of  $\Sigma_N$:

\begin{theorem}\label{covariance-matrix-general}
    For the Hermite and
Laguerre  cases, the covariance matrices $\Sigma_N=(\sigma^N_{i,j})_{i,j=1,\ldots,N}$ are given with the 
notations of Lemma \ref{general-eigenvectors} and with the eigenvalues $\lambda_k$ from the Theorems \ref{ev-a}, \ref{ev-b}, and \ref{ev-main} by
        \begin{align}\label{covarmatwithchristoffel}
       \sigma^N_{i,j} &=\frac{\sqrt{\pi(z_{i,N})\pi(z_{j,N})}}{\kappa_N\tilde{P}_{N-1}(z_{i,N})\tilde{P}_{N-1}(z_{j,N})}
\sum_{k=0}^{N-1}\frac{\tilde{P}_{k}(z_{i,N})\tilde{P}_{k}(z_{j,N})}{\lambda_{N-k}} \notag\\
   &=\frac{(-1)^{i+j}}{\sqrt{\sum_{k,l=0}^{N-1}\tilde{P}_k^2(z_{i,N})\tilde{P}_l^2(z_{j,N})}}
\sum_{k=0}^{N-1}\frac{\tilde{P}_{k}(z_{i,N})\tilde{P}_{k}(z_{j,N})}{\lambda_{N-k}}.
        \end{align}
Moreover, a corresponding result holds in the trigonometric Jacobi case for the covariance matrices
 $\tilde\Sigma_N=(\tilde\sigma^N_{i,j})_{i,j=1,\ldots,N}$. 
\end{theorem}

\begin{proof}
        In all cases, 
 \begin{equation}
        T_N^T\Sigma_N T_N=\operatorname{diag}( \lambda_{1}^{-1},\ldots,\lambda_{N}^{-1}),
        \end{equation}
where the orthogonal matrix $T$ has entries
        \begin{align}
        [T_N]_{i,j}={Q}_{j-1}^{(N)}(z_{i,N})&=\sqrt{\frac{\pi(z_{i,N})}{\kappa_N}}\tilde{Q}_{j-1,N}(z_{i,N})\notag\\
        &=\frac{1}{c_{i,N}}\sqrt{\frac{\pi(z_{i,N})}{\kappa_N}}\tilde{P}_{N-j}(z_{i,N}).
        \end{align}
Hence,
\begin{align}
     \sigma^N_{i,j}=\frac{\sqrt{\pi(z_{i,N})\pi(z_{j,N})}}{\kappa_Nc_{i,N}c_{j,N}}\sum_{k=0}^{N-1}\frac{\tilde{P}_{N-1-k}(z_{i,N})\tilde{P}_{N-1-k}(z_{j,N})}{\lambda_{k+1}}.
        \end{align}
       The substitution $N-1-k\to k$ and
\eqref{cifirstexpression}, \eqref{cisecondexpression}, and \eqref{cilastexpression} yield the result.
\end{proof}

For the convenience of the reader, we restate this result in the Hermite case in terms of the classical Hermite polynomials $H_k$.
We point out that this result was recently also shown by  Gorin and Kleptsyn [GK] by completely different methods.

\begin{corollary}
	In the Hermite case, the covariance matrices $\Sigma_{N}^H=(\sigma^{N,H}_{i,j})_{i,j=1,\ldots,N}$ satisfy
	\begin{align*}
		\sigma^{N,H}_{i,j}=(-1)^{i+j}\left(\sum_{k=0}^{N-1}\frac{(H_k (z_{i,N}))^2}{2^kk!}\sum_{l=0}^{N-1}\frac{(H_l(z_{j,N}))^2}{2^ll!}\right)^{-1/2}
		\sum_{k=0}^{N-1}\frac{H_k(z_{i,N})H_k(z_{j,N})}{2^kk!(N-k)}.
	\end{align*}
\end{corollary}

   The formulas for the entries of the covariance matrices $\Sigma_N$  in (\ref{covarmatwithchristoffel}) 
should be compared with the corresponding results of Dumitriu and Edelman \cite{DE2} for Hermite and Laguerre ensembles.
In the Hermite case, the entries of  $\Sigma_N$  in (\ref{covarmatwithchristoffel}) must be equal to that in (\ref{covariance-a-de})
in Theorem \ref{clt-main-a-DE}. As already pointed out in the introduction, we are not able to verify the equivalence of (\ref{covarmatwithchristoffel}) and 
(\ref{covariance-a-de}) for arbitrary $N$. For small $N$, we checked the equality by a numerical computation.
We thus state  $(\ref{covarmatwithchristoffel})=(\ref{covariance-a-de})$
as a corollary for the orthonormal Hermite poynomials $\tilde{H}_k$.

\begin{corollary}\label{equality-of-variances-h} For $i,j=1,\ldots,N$, the zeros of $\tilde{H}_N$ satisfy
 \begin{align*}
 (-1)^{i+j}\left(\sum_{k=0}^{N-1}{(\tilde{H}_k (z_{i,N}))^2}{}\sum_{l=0}^{N-1}{(\tilde{H}_l(z_{j,N}))^2}\right)^{1/2}
 \sum_{k=0}^{N-1}\frac{\tilde{H}_k(z_{i,N})\tilde{H}_k(z_{j,N})}{N-k}\\
 = \sum_{l=0}^{N-1} \tilde H_l^2(z_{i,N}) \tilde H_l^2(z_{j,N})
 	+ \sum_{l=0}^{N-2}\tilde H_{l+1}(z_{i,N}) \tilde H_l(z_{i,N}) \tilde
 	H_{l+1}(z_{j,N})\tilde H_{l}(z_{j,N}).
 \end{align*}
\end{corollary}

In our opinion, our representation in (\ref{covarmatwithchristoffel}) is slightly nicer than   (\ref{covariance-a-de})
 in the Hermite case.
In the Laguerre case we have a corresponding picture. However, here our formula 
(\ref{covarmatwithchristoffel}) has the same structure as in the Hermite case, while the corresponding formula in \cite{DE2} is much more involved. In the Jacobi case, there do not exist formulas for the entries of $\Sigma_N$ in the literature as far as we are aware.

All preceding results for  $\beta$-Jacobi ensembles were stated in trigonometric coordinates as only in this case the eigenvalues and eigenvectors of the
(inverse) covariance 
matrices of the limit are known; see  Theorem \ref{ev-main}.
 On the other hand, all  trigonometric results above can be easily transfered 
to classical $\beta$-Jacobi ensembles. We briefly collect these results here.
We  follow \cite{F, K, KN, Me, HV} and consider  the
  $\beta$-Jacobi  random matrix ensembles for $k_1,k_2,k_3\ge0$ 
with the joint eigenvalue distributions $\mu_{(k_1,k_2,k_3)}$
with the  densities
\begin{equation}\label{joint-density}
c_{k_1,k_2,k_3}\prod_{1\leq i< j \leq N}\left(x_j-x_i\right)^{k_3}\prod_{i=1}^N
\left(1-x_i\right)^{\frac{k_1+k_2}{2}-\frac{1}{2}}\left(1+x_i\right)^{\frac{k_2}{2}-\frac{1}{2}}
\end{equation}
 on the  alcoves 
$A:=\{x\in\R^N: \> -1\leq x_1\leq ...\leq x_N\leq 1\}$
with some  Selberg constant $c_{k_1,k_2,k_3}>0$.
As in Section 3 we write
$(k_1,k_2,k_3)=\kappa\cdot(a,b,1)$ with
$a\ge0$ $b>0$  fixed and $\kappa\to\infty$. We put $\alpha:=a+b-1>-1$, $\beta=b-1>-1$, and
consider the vector $z:=(z_1, \ldots, z_N)\in A$ consisting of the ordered
zeros of the Jacobi polynomial $P_N^{(\alpha,\beta)})$. Using the 
transformation
$$T: \tilde A\longrightarrow A, \quad T(t_1,\ldots,t_N):=(\cos(2t_1),\ldots,
\cos(2t_N)),$$
the WLT \ref{theoremCLT-transformed} then reads as
follows by \cite{HV}.

\begin{theorem}\label{jacobi-theoremCLT}
Let $a\ge0$ and $b>0$.
Let $X_\kappa$ be random variables with the distributions $\mu_{\kappa\cdot(a,b,1)}$
 as above. 
 Then $\sqrt{\kappa}(X_\kappa-z)$
converges for $\kappa\rightarrow\infty$ to the  normal distribution $N(0,\Sigma_N)$
with some
 regular covariance matrix $\Sigma_N$ whose inverse  
 $\Sigma_N^{-1}=:S_N=(s_{i,j})_{i,j=1,...,N}$ is given by
\begin{align*}
s_{i,j}=
\begin{cases}\sum_{l=1,\ldots,N; l\ne j}
\frac{1}{(z_j-z_l)^2}+\frac{a+b}{2}\frac{1}{(1-z_j)^2}+\frac{b}{2}\frac{1}{(1+z_j)^2}
&\textit{ for }i=j\\
\frac{-1}{(z_i-z_j)^2}&\textit{ for }i\neq j
\end{cases}.
\end{align*}
\end{theorem}

The inverse covariance matrices  $\tilde\Sigma_N^{-1}$ and  $\Sigma_N^{-1}$ from the WLTs
\ref{theoremCLT-transformed} and \ref{jacobi-theoremCLT} are related by
$\tilde S=DSD$ with the diagonal matrix 
$$D=\operatorname{diag}{\left(-2\sqrt{1-z_{1,N}^2},\ldots,-2\sqrt{1-z_{N,N}^2}\right)}$$
by \cite{HV}. Hence, Theorem \ref{covariance-matrix-general} means in the non-trigonometric Jacobi case:

\begin{theorem}\label{covariance-matrix-general-non-trigonometric }
   The covariance matrix $\Sigma_N=(\sigma^N_{i,j})_{i,j=1,\ldots,N}$ in Theorem \ref{jacobi-theoremCLT} has entries
\begin{align}\label{covarmatwithchristoffel-non-trigonometric}
       \sigma^N_{i,j} =\frac{(-1)^{i+j}4 \sqrt{1-z_{i,N}^2}\sqrt{1-z_{j,N}^2}}{\sqrt{\sum_{k,l=0}^{N-1}(\tilde{P}_k^{(\alpha,\beta)}(z_{i,N})\tilde{P}_l^{(\alpha,\beta)}(z_{j,N}))^2}}
\sum_{k=0}^{N-1}\frac{\tilde{P}_{k}^{(\alpha,\beta)}(z_{i,N})\tilde{P}_{k}^{(\alpha,\beta)}(z_{j,N})}{\lambda_{N-k}},
        \end{align}
        with $\lambda_k$ as specified in Theorem~\ref{ev-main}.
\end{theorem}

\section{Limit results for the largest  eigenvalue for $N\to\infty$ in the Hermite case}\label{chapterhermitesoftedge}

In this section we discuss the soft edge statistics in the Hermite case in the freezing regime.
This means that we analyze the limit behavior of the largest eigenvalue in the freezing regime in Theorem \ref{ev-a}
for $N\to\infty$. This will be done on the basis of Theorem \ref{covariance-matrix-general}.
We remark that this problem is also discussed in \cite{DE2} through (\ref{covariance-a-de}). We show
that  our approach via (\ref{covarmatwithchristoffel}) leads to a limit  with a different form from that in \cite{DE2}.

As in Section 2, let now  $z_{1,N}<...,z_{N,N}$ be the the ordered zeros  of the Hermite polynomial $H_N$.
Moreover, for each $N$, let
 $(Q_{k,N})_{k=0,..,N-1}$ be the dual polynomials associated with $(H_k)_{k=0,...,N}$  normalized as in 
(\ref{normalization-ops-a}). This means that $T_N:= (Q_{j-1,N}(z_{i,N}))_{i,j=1,\ldots,N}$ is an orthogonal matrix 
with $T_N^T\Sigma_N
T_N=\operatorname{diag}(1,...,\frac{1}{N})$ as in  the proof of Theorem \ref{covariance-matrix-general}.
These polynomials satisfy the three-term-recurrence
\begin{align}\label{3term}
        xQ_{k,N}(x)=\sqrt{\frac{N-k-1}{2}}Q_{k+1,N}(x)+\sqrt{\frac{N-k}{2}}Q_{k-1,N}(x) \quad(k\le N)
\end{align}
with the initial conditions $Q_{-1,N}=0$ and $Q_{0,N}=\frac{1}{\sqrt{N}}$.

We now derive a limit result for  $N\to\infty$ which involves the Airy function  $\mathsf {Ai}$.
For this we recapitulate some well known facts about  $\mathsf {Ai}$; see e.g.~Section 9 of \cite{NIST} or the monograph \cite{VS}.
 $\mathsf {Ai}$ is the unique solution of
\begin{equation}\label{ode-airy} y^{\prime\prime}(z)=z\cdot y(z) \quad (z\in\mathbb R) \quad\text{with}\quad \lim_{z\to\infty} y(z)=0 
\end{equation}
and with $y(0)=\frac{1}{3^{2/3}\Gamma(2/3)}=0.355028\ldots$.
The Airy function   $\mathsf {Ai}$ has a unique largest zero at
  $a_1=-2.338\ldots$ with $\mathsf {Ai}(z)>0$ for $z> a_1$.
Moreover,  $\mathsf {Ai}$ has infinitely many isolated, simple zeros in $]-\infty,a_1]$. 
 For $r\in\mathbb N$, the $r$-th largest zero $a_r$ of $\mathsf {Ai}$ satisfies 
\begin{equation}\label{airy-zero-asymtot}
a_r\simeq -\Bigl(\frac{3\pi}{2}(r-1/4)\Bigr)^{2/3} \quad\quad\text{for}\quad r\to\infty.
\end{equation}
In addition, we have the asymptotic behavior as $z\to-\infty$
\begin{equation}\label{airy-minus-inf}
\mathsf{Ai}(-z)\simeq\frac{1}{\sqrt{\pi}z^{1/4}}\cos\Big(\frac{2}{3}z^{3/2}-\frac{\pi}{4}\Big),
\end{equation}
as well as
\begin{equation}\label{airy-der-at-zero}
\mathsf{Ai}'(a_r)\simeq\frac{(-1)^{r-1}}{\sqrt{\pi}}\Big(\frac{3\pi}{2}(r-1/4)\Big)^{1/6}\quad\quad\text{for}\quad r\to\infty.
\end{equation}

The following theorem is the central step for our limit results for $N\to\infty$:

\begin{theorem}\label{localconvfn}
        Consider the functions 
        \begin{align*}
        f_N(y):=N^\frac{1}{6}Q_{\lfloor N^\frac{1}{3}y\rfloor,N}(z_{N,N}) \quad\text{for}\quad y\in[0, N^\frac{2}{3}[
        \end{align*}
and  $f_N(y)=0$ otherwise.
        Then $(f_N)_{N\ge1}$ tends for $N\to\infty$ locally uniformly to 
        \begin{align*}
        f(y)=\frac{\mathsf {Ai}(y+a_1)}{\mathsf {Ai}'(a_1)} \quad\quad\text{for}\quad y\in[0,\infty[.
        \end{align*}
\end{theorem}

We  split the proof   into three lemmas and use the abbreviation 
$q_k:=Q_{k,N}(z_{N,N})$ where we suppress the dependence on $N$. We start with the following result:

\begin{lemma}\label{lemmaODEfN}
        The functions $f_N$ satisfy for $y\in[0,N^\frac{2}{3}[$ the equation
        \begin{align*}
        f_N(y)=\int_{0}^y\int_{0}^{s}(t-|a_1|)f_N(t)\> dt\> ds+y+\operatorname{err}(y,N). 
        \end{align*}
       The error term $\operatorname{err}(y,N)$ is specified in Eq.~(\ref{errorterm1}) at the end of the proof.
\end{lemma}

\begin{proof} Let $y\ge0$. We divide the recurrence (\ref{3term}) with $x:=z_{N,N}$ by $\sqrt{N}$ and  get 
        \begin{align}\label{3termscaled}
        \sqrt{1-\frac{k+1}{N}}q_{k+1}=\frac{2z_{N,N}}{\sqrt{2N}}q_k-\sqrt{1-\frac{k}{N}}q_{k-1}\quad(k\le N)
        \end{align}
     with $q_{-1}=0$, $q_0=N^{-1/2}$.  We next observe that by the  Lagrange remainder in Taylor's formula,
 for $k=0,...,\lfloor yN^\frac{1}{3}\rfloor$,
\begin{equation}\label{Lagrange} \sqrt{1-\frac{k}{N}}=1-\frac{k}{2N}-\frac{1}{8(1-\xi_{k})^\frac{3}{2}}\left(\frac{k}{N}\right)^2
\quad \text{with}\quad  \xi_k\in (0,\frac{k}{N}).
\end{equation}
               We now  define
        \begin{align*}
        \alpha(k,N):=\frac{1}{8(1-\xi_{k})^\frac{3}{2}}\left(\frac{k}{N}\right)^2
        \end{align*}
        and conclude from (\ref{Lagrange}) that for $k=0,...,\lfloor yN^\frac{1}{3}\rfloor$
        \begin{align}\label{alphaestimates}
        0        <\alpha(k,N)<\left(\frac{N}{N-k}\right)^\frac{3}{2}\left(\frac{k}{N}\right)^2.
        \end{align}
 Moreover, we obtain from
  a sharp  Plancherel-Rotach   theorem of Ricci  \cite{Ric} that
        \begin{equation}\label{Plancherel-Rotach} \frac{z_{N,N}}{\sqrt{2N}}=1-\frac{|a_1|}{2N^\frac{2}{3}}+O(N^{-1}).
\end{equation}
  Using (\ref{Plancherel-Rotach}) we  rewrite the  recurrence    (\ref{3termscaled}) as 
        \begin{align}\label{eq3termapproximated}
        &q_{k+1}-q_k-\left(q_k-q_{k-1}\right)\\=&\frac{k+1}{2N}q_{k+1}-\frac{|a_1|}{N^\frac{2}{3}}q_k+
\frac{k}{2N}q_{k-1}+\alpha(k+1,N)q_{k+1}+\alpha(k,N)q_{k-1}+O(N^{-1})q_k.
        \notag\end{align}
        Summation over $k=0,...,l$ now yields
        \begin{align*}
        q_{l+1}-q_{l}-\frac{1}{\sqrt{N}}
&=q_{l+1}-q_{l}-(q_0-q_{-1})=\sum_{k=0}^{l}\Big(q_{k+1}-q_k-\left(q_k-q_{k-1}\right)\Big)\\
        &=\sum_{k=0}^{l}\Big(\frac{k+1}{2N}q_{k+1}-\frac{|a_1|}{N^\frac{2}{3}}q_k+\frac{k}{2N}q_{k-1}\Big)+
\\ &\quad +\sum_{k=0}^{l}\Big(\alpha(k+1,N)q_{k+1}+\alpha(k,N)q_{k-1}+O(N^{-1})q_k\Big).
        \end{align*}
        A second summation over $l=0,...,\lfloor yN^\frac{1}{3}\rfloor-1$ now leads to
        \begin{align*}
        &q_{\lfloor y N^\frac{1}{3}\rfloor}-\frac{\lfloor
yN^\frac{1}{3}\rfloor+1}{\sqrt{N}}=\sum_{l=0}^{\lfloor
yN^\frac{1}{3}\rfloor-1}\left(q_l-q_{l-1}-\frac{1}{\sqrt{N}}\right)=
        \\&=\sum_{l=0}^{\lfloor
yN^\frac{1}{3}\rfloor-1}\sum_{k=0}^{l}\left(\frac{k+1}{2N}q_{k+1}-\frac{|a_1|}{N^\frac{2}{3}}q_k+\frac{k}{2N}q_{k-1}\right)
        +\rho(y,N) \end{align*}
with 
\begin{equation}\label{def-rho}
 \rho(y,N):=
\sum_{l=0}^{\lfloor
yN^\frac{1}{3}\rfloor-1}\sum_{k=0}^{l}\Big(\alpha(k+1,N)q_{k+1}+\alpha(k,N)q_{k-1}+O(N^{-1})q_k\Big).
        \end{equation}
        If we multiply this by $N^\frac{1}{6}$ we get
        \begin{align}\label{discreteinteq}
        &f_N(y)-\frac{\lfloor
yN^\frac{1}{3}\rfloor+1}{N^\frac{1}{3}}-N^\frac{1}{6}\rho(y,N)\notag\\
        &=\frac{1}{N^\frac{1}{3}}\sum_{l=0}^{\lfloor
yN^\frac{1}{3}\rfloor-1}\frac{1}{N^\frac{1}{3}}\sum_{k=0}^{l}\left(\frac{k+1}{2N^\frac{1}{3}}N^\frac{1}{6}q_{k+1}-
{|a_1|}N^\frac{1}{6}q_k+\frac{k}{2N^\frac{1}{3}}N^\frac{1}{6}
q_{k-1}\right)\notag\\
        &=\frac{1}{N^\frac{1}{3}}\sum_{l=0}^{\lfloor
yN^\frac{1}{3}\rfloor-1}\frac{1}{N^\frac{1}{3}}\sum_{k=0}^{l}\left(N^\frac{1}{6}q_k\left(\frac{k+\frac{1}{2}}{N^\frac{1}{3}}-|a_1|\right)\right)+
\notag\\
&\quad+\frac{1}{N^\frac{2}{3}}\sum_{l=0}^{\lfloor
yN^\frac{1}{3}\rfloor-1}\left(\frac{l+1}{2N^\frac{1}{3}}N^\frac{1}{6}\left(q_{l+1}-q_l\right)\right).
        \end{align}
        Notice that  the last equation was obtained from the  shifts  $k+1\mapsto k$ and
$k-1\mapsto k$.
We now compare the r.h.s.~of (\ref{discreteinteq}) with
        \begin{equation}\label{def-gN}
        \int_{0}^{y}\int_{0}^{s}(x-|a_1|)f_N(t)dtds.
        \end{equation}
      For this  we use the functions
        \begin{align*}
        g_N(t):=\sum_{k=0}^{N-1}t_{k,N}\mathbf 1_{\left[\frac{k}{N^\frac{1}{3}},\frac{k+1}{N^\frac{1}{3}}\right]}(t)
\quad\text{ with}\quad t_{k,N}:=N^\frac{1}{6}q_k\left(\frac{k+\frac{1}{2}}{N^\frac{1}{3}}-|a_1|\right).
        \end{align*}
An elementary calculation yields
   \begin{align}   
\int_{0}^{y}&\int_{0}^{s} g_N(t)\> dtds=
        \int_{0}^{y}\int_{0}^{s}\sum_{k=0}^{N-1}t_{k,N}\mathbf{1}_{\left[\frac{k}{N^\frac{1}{3}},\frac{k+1}{N^\frac{1}{3}}\right]}(t)\> dt \> ds\notag \\
&=\frac{1}{2}\left(\frac{yN^\frac{1}{3}-\lfloor
yN^\frac{1}{3}\rfloor}{N^\frac{1}{3}}\right)^2t_{\lfloor y N^\frac{1}{3}\rfloor,N}
+\frac{1}{N^\frac{2}{3}}\sum_{k=0}^{\lfloor y N^\frac{1}{3}\rfloor-1}(\lfloor y N^\frac{1}{3}\rfloor-k)t_{k,N}\notag \\
&\quad\quad
+\frac{yN^\frac{1}{3}-\lfloor
yN^\frac{1}{3}\rfloor-\frac{1}{2}}{N^\frac{1}{3}}\cdot \frac{1}{N^\frac{1}{3}}\sum_{k=0}^{\lfloor
y N^\frac{1}{3}\rfloor -1}t_{k,N}.
\notag   \end{align}      
 Moreover,
        \begin{align}\label{Abelpartsum}
        \sum_{l=0}^{L}\sum_{k=0}^{l}t_{k,N}=\sum_{k=0}^{L}(L-k+1)t_{k,N} \quad\quad (L\in \mathbb{N}).
        \end{align}
Hence,
        \begin{align}\label{discretedoubleint}
        &\int_{0}^{y}\int_{0}^{s}g_N(t)dtds-\frac{1}{N^\frac{2}{3}}\sum_{l=0}^{\lfloor y
N^\frac{1}{3}\rfloor-1}\sum_{k=0}^{l}t_{k,N}\\
        &=\frac{1}{2}\left(\frac{yN^\frac{1}{3}-\lfloor
yN^\frac{1}{3}\rfloor}{N^\frac{1}{3}}\right)^2t_{\lfloor y N^\frac{1}{3}\rfloor,N}
        +\frac{yN^\frac{1}{3}-\lfloor
yN^\frac{1}{3}\rfloor-\frac{1}{2}}{N^\frac{1}{3}}\cdot \frac{1}{N^\frac{1}{3}}\sum_{k=0}^{\lfloor
y N^\frac{1}{3}\rfloor -1}t_{k,N}.
        \notag\end{align}
       (\ref{discreteinteq}), (\ref{def-gN}), and (\ref{discretedoubleint}) now show that
        \begin{align}\label{inteqfirstvers}
        f_N(y)=\int_{0}^{y}\int_{0}^{s}f_N(t)\left(\frac{\lfloor t N^\frac{1}{3}\rfloor
+\frac{1}{2}}{N^\frac{1}{3}}-|a_1|\right)dtds+\frac{\lfloor
yN^\frac{1}{3}\rfloor+1}{N^\frac{1}{3}}+\widetilde{\operatorname{err}}(N,y)
        \end{align}
        with the error term
        \begin{align*}
        \widetilde{\operatorname{err}}(N,y)&:=N^\frac{1}{6}\rho(y,N)+\frac{1}{N^\frac{2}{3}}\sum_{l=0}^{\lfloor
yN^\frac{1}{3}\rfloor-1}\frac{l+1}{2N^\frac{1}{3}}\left(q_{l+1}-q_l\right)\\
        &-\frac{yN^\frac{1}{3}-\lfloor
yN^\frac{1}{3}\rfloor-\frac{1}{2}}{N^\frac{1}{3}}\frac{1}{N^\frac{1}{3}}\sum_{k=0}^{\lfloor
y N^\frac{1}{3}\rfloor
-1}N^\frac{1}{6}q_k\left(\frac{k+\frac{1}{2}}{N^\frac{1}{3}}-|a_1|\right)\\
        &-\frac{1}{2}\left(\frac{yN^\frac{1}{3}-\lfloor
yN^\frac{1}{3}\rfloor}{N^\frac{1}{3}}\right)^2N^\frac{1}{6}q_{\lfloor
yN^\frac{1}{3}\rfloor}\left(\frac{\lfloor
yN^\frac{1}{3}\rfloor+\frac{1}{2}}{N^\frac{1}{3}}-|a_1|\right).
        \end{align*}
        As
        \begin{align*}
        \frac{\lfloor yN^\frac{1}{3}\rfloor+1}{N^\frac{1}{3}}=y+O(N^{-\frac{1}{3}})
        \end{align*}
        and 
        \begin{align*}
        \frac{\lfloor tN^\frac{1}{3}\rfloor+\frac{1}{2}}{N^\frac{1}{3}}=t+\frac{\lfloor
tN^\frac{1}{3}\rfloor-tN^\frac{1}{3}+\frac{1}{2}}{N^\frac{1}{3}}=t+O(N^{-\frac{1}{3}}),
        \end{align*}
        we get
        \begin{align*}
        f_N(y)=\int_{0}^y\int_{0}^{s}(t-|a_1|)f_N(t)dtds+y+\operatorname{err}(y,N)
        \end{align*}
        with the error term
        \begin{align}\label{errorterm1}
        \operatorname{err}(N,y)&=N^\frac{1}{6}\rho(y,N)+\frac{1}{N^\frac{2}{3}}\sum_{l=0}^{\lfloor
yN^\frac{1}{3}\rfloor-1}\frac{l+1}{2N^\frac{1}{3}}\left(q_{l+1}-q_l\right)\notag\\
        &-\frac{yN^\frac{1}{3}-\lfloor
yN^\frac{1}{3}\rfloor-\frac{1}{2}}{N^\frac{1}{3}}\frac{1}{N^\frac{1}{3}}\sum_{k=0}^{\lfloor
y N^\frac{1}{3}\rfloor
-1}N^\frac{1}{6}q_k\left(\frac{k+\frac{1}{2}}{N^\frac{1}{3}}-|a_1|\right)\notag\\
        &-\frac{1}{2}\left(\frac{yN^\frac{1}{3}-\lfloor
yN^\frac{1}{3}\rfloor}{N^\frac{1}{3}}\right)^2N^\frac{1}{6}q_{\lfloor
yN^\frac{1}{3}\rfloor}\left(\frac{\lfloor yN^\frac{1}{3}\rfloor
                +\frac{1}{2}}{N^\frac{1}{3}}-|a_1|\right)\notag\\
        &+\frac{\lfloor y N^\frac{1}{3}\rfloor-yN^\frac{1}{3}+1}{N^\frac{1}{3}}
        +\int_{0}^{y}\int_{0}^{s}\frac{\lfloor t
N^\frac{1}{3}\rfloor-tN^\frac{1}{3}+\frac{1}{2}}{N^\frac{1}{3}}f_N(t)dtds
        \end{align}
\end{proof}

\begin{lemma}\label{lemmaerrbound}
        The error term in (\ref{errorterm1}) satisfies
        $\operatorname{err}(N,y)=O(N^{-\frac{1}{3}})$
         locally uniformly in $y\in[0,\infty[$.
\end{lemma}

\begin{proof} Fix some $M>0$ and consider $y\in[0,M]$. We recapitulate that  the matrices 
         $T_N=(Q_{k-1,N}(z_{i,N}))_{k,i=1,...,N}$ are orthogonal which implies that
        for all $N\in\mathbb{N}$
        \begin{align}\label{rowsumeq}
        1=\sum_{k=0}^{N-1}(Q_{k,N}(z_{N,N}))^2=\frac{1}{N^\frac{1}{3}}\sum_{k=0}^{N-1}(N^\frac{1}{6}Q_{k,N}(z_{N,N}))^2
        =\int_{0}^{\infty}f_N^2(t)dt.
        \end{align}
        We next prove
        \begin{align}\label{intfnlocalbound}
        \int_{0}^{y}f_N(t)dt=O(1) \quad\text{for}\quad N\to\infty.
        \end{align}                
        For this we recall that by the definition of $f_N$
        \begin{align*}
        f_N(t)=\sum_{k=0}^{N-1}N^\frac{1}{6}Q_{k,N}(z_{N,N})\mathbf{1}_{\left[\frac{k}{N^{1/3}},\frac{k+1}{N^{1/3}}\right]}(t)
        \end{align*}
       with $Q_{k,N}(z_{N,N})>0$ for all $k$. This follows from the fact that the polynomials $Q_{k,N}$
 have a positive leading coefficient and are orthogonal w.r.t.~some measure with support $\{z_{1,N},\ldots,z_{N,N}\}$ which implies that 
all their zeros are contained in $]z_{1,N},z_{N,N}[$; see e.g.~\cite{C}.
We thus see that $f_N(t)\geq 0$ for $t\geq 0$. Hence, for $y\in[0,M]$,
        \begin{align*}
        \int_{0}^{y}f_N(t)dt&\leq\int_{0}^{M}f_N(t)dt
=\frac{1}{N^\frac{1}{3}}\sum_{k=0}^{\lfloor
MN^\frac{1}{3}\rfloor-1}N^\frac{1}{6}q_k
        +\frac{MN^\frac{1}{3}-\lfloor
MN^\frac{1}{3}\rfloor}{N^\frac{1}{3}}N^\frac{1}{6}q_{\lfloor
MN^\frac{1}{3}\rfloor}\\
        &\leq\frac{1}{N^\frac{1}{3}}\sum_{k=0}^{\lfloor
MN^\frac{1}{3}\rfloor}N^\frac{1}{6}q_k.
        \end{align*}
        H\"older's inequality and \eqref{rowsumeq} now imply that for $y\in[0,M]$ and $N\in\mathbb{N}$,
        \begin{align}\label{intfnlocalboundconcrete}
        \int_{0}^{y}f_N(t)dt&\leq\frac{1}{N^\frac{1}{3}}\left(\sum_{k=0}^{\lfloor
MN^\frac{1}{3}\rfloor}q_k^2\right)^\frac{1}{2}
        \left(\sum_{k=0}^{\lfloor MN^\frac{1}{3}\rfloor}N^\frac{1}{3}\right)^\frac{1}{2}\\
        &\leq\frac{1}{N^\frac{1}{3}}\sqrt{N^\frac{1}{3}(\lfloor
MN^\frac{1}{3}\rfloor+1)}\leq \sqrt{M+\frac{2}{N^\frac{1}{3}}}\leq
\sqrt{M+2}.\notag
        \end{align}
      This shows \eqref{intfnlocalbound}. In an analogous way we prove that for $y\in[0,M]$ and $\theta\in[0,1]$,
        \begin{align}\label{intyfnlocalbound}
        \frac{1}{N^\frac{1}{3}}\sum_{l=0}^{\lfloor
yN^\frac{1}{3}\rfloor-1}\frac{l+\theta}{N^\frac{1}{3}}N^\frac{1}{6}q_l=O(1).
        \end{align}
        For this  we  observe that
        \begin{align*}
        \sum_{l=0}^{\lfloor
yN^\frac{1}{3}\rfloor-1}\frac{l+\theta}{N^\frac{1}{3}}N^\frac{1}{6}q_l\leq
        \sum_{l=0}^{\lfloor
MN^\frac{1}{3}\rfloor-1}\frac{MN^\frac{1}{3}+1}{N^\frac{1}{3}}N^\frac{1}{6}q_l\leq
        {(M+1)}\sum_{l=0}^{\lfloor MN^\frac{1}{3}\rfloor-1}N^\frac{1}{6}q_l.
        \end{align*}
        This together with (\ref{intfnlocalbound}) shows (\ref{intyfnlocalbound}).

         Moreover, (\ref{intfnlocalboundconcrete}) leads to the following estimate for the last term in 
(\ref{errorterm1}):
\begin{align}\label{errbounddoubleint}
        \left|\int_{0}^{y}\int_{0}^{s}\frac{\lfloor
tN^\frac{1}{3}\rfloor-tN^\frac{1}{3}+\frac{1}{2}}{N^\frac{1}{3}}f_N(t)\>dt\>ds\right|&\leq
        \frac{1}{2N^\frac{1}{3}}\int_{0}^{y}\sqrt{M+2}\>ds\notag\\
        &\leq\frac{M\sqrt{M+2}}{N^\frac{1}{3}}=O(N^{-\frac{1}{3}}).
        \end{align}
        We  now turn to the estimation of $N^\frac{1}{6}\rho(y,N)$. For $y\in[0,M]$ and 
$k=0,...,\lfloor yN^\frac{1}{3}\rfloor $ we obtain that $N/(N-k)$ remains bounded for large $N$. Therefore,
 (\ref{alphaestimates}) implies readily that $\alpha(k,N)=O(N^{-\frac{4}{3}})$ and thus, by (\ref{def-rho}),
        \begin{align}\label{estimates-rho}
        |&N^\frac{1}{6}\rho(y,N)|\\ &\leq \sum_{l=0}^{\lfloor
yN^\frac{1}{3}\rfloor-1}
\sum_{k=0}^{l}\left(|O(N^{-\frac{4}{3}})|N^\frac{1}{6}q_{k+1}+|O(N^{-\frac{4}{3}})|N^\frac{1}{6}q_{k-1}+|O(N^{-1})|N^\frac{1}{6}q_k\right)\notag\\
        &\leq \sum_{l=0}^{\lfloor M
N^\frac{1}{3}\rfloor-1}\sum_{k=0}^{l}\left(|O(N^{-\frac{7}{6}})|q_{k+1}+|O(N^{-\frac{7}{6}})|q_{k-1}+|O(N^{-\frac{5}{6}})|q_k\right).
               \notag\end{align}
        If we use the summation formula (\ref{Abelpartsum}) and H\"older's inequality,
 we see that the third summand on the r.h.s.~of (\ref{estimates-rho}) satisfies
        \begin{align*}
        &|O(N^{{-\frac{5}{6}}})|\sum_{l=0}^{\lfloor
MN^\frac{1}{3}\rfloor-1}\sum_{k=0}^{l}q_k=|O(N^{-5/6})|\sum_{k=0}^{\lfloor M
N^\frac{1}{3}\rfloor -1} (\lfloor M N^\frac{1}{3}\rfloor -k)q_k\\
        \leq &|O(N^{{-\frac{5}{6}}})|\left(\sum_{k=0}^{\lfloor M N^\frac{1}{3}\rfloor-1}
q_k^2\right)^\frac{1}{2}
        \left(\sum_{k=0}^{\lfloor M N^\frac{1}{3}\rfloor-1} \underbrace{(\lfloor M
N^\frac{1}{3}\rfloor -k)^2}_{\leq M^2N^\frac{2}{3}}\right)^\frac{1}{2}\\
        \leq &|O(N^{{-\frac{5}{6}}})|  \left(\lfloor M N^\frac{1}{3}\rfloor M^2
N^\frac{2}{3}\right)^\frac{1}{2}=O(N^{-\frac{1}{3}}). 
        \end{align*}
        If we keep in mind that $q_0=\frac{1}{\sqrt{N}}$ we can estimate the other two sums
in the same way. In summary, we conclude for the first term in (\ref{errorterm1}) that
\begin{align*}
        N^\frac{1}{6}\rho(y,N)=O(N^{-\frac{1}{3}}).
        \end{align*}
Furthermore, the  second term in (\ref{errorterm1}) can be estimated by a corresponding bound by
(\ref{intyfnlocalbound}) with  $\theta=1/2$ and $1$ and with an index
shift together with $Q_{0,N}=\frac{1}{\sqrt{N}}$.
 Moreover, the third term in (\ref{errorterm1}) can be estimated in the same way by
splitting the sum there and using (\ref{intyfnlocalbound}) for the first and
(\ref{rowsumeq}) for the second sum. Finally, the fourth and fifth term  in (\ref{errorterm1})
 obviously have order $O(N^{-\frac{1}{3}})$, while this follows for the  last term from   (\ref{errbounddoubleint}).
This completes the proof.
\end{proof}

We now complete the proof of Theorem \ref{localconvfn} by proving the following

\begin{lemma}\label{lemmalocalconvN13}
 For $N\to\infty$, 
        $|f_N(y)-f(y)|=O(N^{-\frac{1}{3}})$ locally uniformly for $y\in[0,\infty[$.
\end{lemma}

\begin{proof} Again, fix $M>0$, let $y\in [0,M]$, and assume that $N^\frac{2}{3}>M$. 
The  ODE (\ref{ode-airy}) yields that the function $f(y)=\frac{\mathsf {Ai}(y+a_1)}{\mathsf {Ai}'(a_1)}$
satisfies
 \begin{equation}\label{prop-airy}
        f''(y)=(y+a_1)f(y) \quad\text{with}\quad f(0)=0,\quad f^\prime(0)=1.
        \end{equation}
This ODE leads to the integral equation
        \begin{equation}\label{partial-integration}
        f(y)=\int_{0}^{y}\int_{0}^{s}(t-|a_1|)f(t)dtds+y=\int_{0}^{y}(t-|a_1|)(y-t)f(t)\>dt+y.
        \end{equation}
       Notice that the second equation in (\ref{partial-integration}) follows by partial integration. Moreover, by Lemma \ref{lemmaODEfN},
        \begin{align*}
        f_N(y)=\int_{0}^{y}(t-|a_1|)(y-t)f_N(t)\> dt+y+\operatorname{err}(y,N).
        \end{align*}
We thus obtain
        \begin{align*}
        |f(y)-f_N(y)|&=\left|\int_{0}^{y}(t-|a_1|)(y-t)(f(t)-f_N(t))dt-\operatorname{err}(y,N)\right|\\
        &\leq\int_{0}^{y}|t-|a_1||\cdot|y-t|\cdot|f(t)-f_N(t)|dt+{|\operatorname{err}(y,N)|}
        \end{align*}
        where we know from Lemma \ref{lemmaerrbound}  that there exists a constant $M'=M'(M)>0$ with
        \begin{align*}
        |\operatorname{err}(y,N)|\leq\frac{M'}{N^\frac{1}{3}} \quad\text{for}\quad y\in[0,M]
        \end{align*}
         and $N$ sufficently large.

As $t\mapsto |(t-|a_1|)(y-t)|$ is the absolute value of a second-order polynomial, we find a constant $M''>0$ with
        $|(t-|a_1|)(y-t)|<M''$ for all $t\in[0,y]$ and $y\in[0,M]$.
        Hence,
        \begin{align*}
        |f(y)-f_N(y)|\leq \int_{0}^{y}M''|f(t)-f_N(t)|dt+\frac{M'}{N^\frac{1}{3}}.
        \end{align*}
 The Lemma of Gronwall now implies our claim that
        \begin{align*}
        |f(y)-f_N(y)|\leq\frac{M'}{N^\frac{1}{3}}e^{M''y}\leq
\frac{M'}{N^\frac{1}{3}}e^{M''M} =O(N^{-\frac{1}{3}}).
        \end{align*}
        \end{proof}

We now apply Lemma \ref{lemmalocalconvN13} to the $(N,N)$-entries of the covariance matrices $\Sigma_{N}$
for $\beta$-Hermite ensembles in the freezing regime as in Theorem \ref{covariance-matrix-general} for $N\to\infty$.

\begin{theorem}\label{corrsigmaasymptotics}
        Consider the  covariance matrices $\Sigma_{N}=:\left(\sigma_{i,j}\right)_{i,j=1,\ldots,N}$ of 
 $\beta$-Hermite ensembles in the freezing regime. Then
        \begin{align*}
        \lim\limits_{N\rightarrow\infty}{N^\frac{1}{3}}{\sigma_{N,N}}=
        \int_{0}^{\infty}\frac{\mathsf {Ai}(x+a_1)^2}{\mathsf {Ai}'(a_1)^2x}dx =0.834\ldots
        \end{align*}
\end{theorem}

\begin{proof}
        We recapitulate that $\Sigma_N =T_N\operatorname{diag}(1,1/2,...,1/N)T_N^T$. Therefore,
        \begin{align*}
        \sigma_{N,N}&=\sum_{k=1}^{N}\frac{1}{k}(Q_{k-1}^N(z_{N,N}))^2
        =\frac{1}{N^\frac{2}{3}}\sum_{k=0}^{N-1}\frac{N^\frac{1}{3}}{k+1}\left(N^\frac{1}{6}Q_k^N(z_{N,N})\right)^2.
        \end{align*}
 Define the functions
        \begin{align*}
        h_N(y):=\sum_{k=0}^{N-1}\frac{N^\frac{1}{3}}{k+1}\mathbf{1}_{\left[\frac{k}{N^\frac{1}{3}},\frac{k+1}{N^\frac{1}{3}}\right)}(y),
        \end{align*}
        which  are  approximations of the function  $y\mapsto\frac{1}{y}$ with
        \begin{align}\label{hNestimates}
        0\leq \frac{1}{y}-h_N(y)\leq
\frac{N^\frac{1}{3}}{k(k+1)}\leq \frac{1}{y}\frac{1}{k}  \quad\text{for}\quad k=\lfloor y N^\frac{1}{3}\rfloor, \>\> y>0.
        \end{align} 
 With this notation we have
        \begin{align*}
    N^\frac{1}{3}  \sigma_{N,N} = \frac{1}{N^\frac{1}{3}}\sum_{k=0}^{N-1}\frac{N^\frac{1}{3}}{k+1}\left(N^\frac{1}{6}Q_k^N(z_{N,N})\right)^2
        =\int_{0}^{\infty}(f_N(y))^2h_N(y)dy.
        \end{align*}
        The statement of the theorem is now equivalent to 
        \begin{align*}
      \lim_{N\to\infty}  \int_{0}^{\infty}(f_N(y))^2h_N(y)dy =\int_{0}^{\infty}\frac{\mathsf
{Ai}(x+a_1)^2}{\mathsf {Ai}'(a_1)^2y}dy=\int_{0}^{\infty}\frac{f(y)^2}{y}dy .
        \end{align*}
       To show this, we prove that
        \begin{align}\label{convint01fn}
        \lim_{N\to\infty}  \int_{0}^{1}f_N(y)^2h_N(y)\>dy=\int_{0}^{1}\frac{f(y)^2}{y}\>dy
        \end{align}
        and 
        \begin{align}\label{convint1inffn}
         \lim_{N\to\infty} \int_{1}^{\infty}f_N(y)^2h_N(y)\>dy = \int_{1}^{\infty}\frac{f(y)^2}{y}\>dy.
        \end{align}
        For this we first recapitulate from \eqref{rowsumeq} that
        \begin{align}\label{eqintfn2}
        \int_{0}^{\infty}(f_N(y))^2\>dy=1.
        \end{align}
        Furthermore, as $f''(y)=(y-|a_1|)f(y)$, we know that
        \begin{align}\label{eqintf2}
        \int_{0}^{\infty}{f(y)^2}\> dy=-\left[|a_1|f(y)^2+f'(y)^2\right]_{y=0}^\infty=1.
        \end{align}
  We next observe that  Theorem \ref{localconvfn} implies that the measures
        $ f_N^2d\lambda$ with Lebesgue densities  $ f_N^2$ converge in a vague way to the measure
 $f^2d\lambda$ on $[0,\infty[$. As all these measures  are probability measures
by  \eqref{eqintfn2} and \eqref{eqintf2}, we conclude from a standard result in probability (see for example \cite{Bi}) that 
these measures  converge even weakly, that is, for all  bounded continuous functions $g:[0,\infty[\to\mathbb R$ we have
        \begin{equation}\label{conv-weakly-0}
        \lim_{N\to \infty}\int_{0}^{\infty}g(y)f_N(y)^2\> dy= \int_{0}^{\infty}g(y)f(y)^2\> dy.
        \end{equation}
Moreover, as all these  probability measures have  Lebesgue densities, 
we again conclude from a standard result in probability (again, see \cite{Bi}) that (\ref{conv-weakly-0}) remains correct
on  $[1,\infty[$, namely, for the   bounded continuous function  $g(y):=\frac{1}{y}$ on $[1,\infty[$
 we have
        \begin{align}\label{weakconvfrom1}
       \lim_{N\to \infty} \int_{1}^{\infty}\frac{f_N(y)^2}{y}\>dy=\int_{1}^{\infty}\frac{f(y)^2}{y}\>dy=:R.
        \end{align} 
On the other hand, \eqref{hNestimates} shows that
  for any $\varepsilon>0$ there is some sufficiently  large  $N(\varepsilon)$  with
\begin{align*} \Bigl|\frac{1}{y}-h_N(y)\Bigr|\le \frac{\varepsilon}{y} \quad\text{for}\quad y\ge 1, \> N\ge N(\varepsilon).
\end{align*}
Therefore,
\begin{equation}\label{estimate-error-n}
\int_{1}^{\infty}  \Bigl|\frac{1}{y}-h_N(y)\Bigr|\>f_N(y)^2\> dy\le \varepsilon\int_{1}^{\infty} \frac{1}{y}\>f_N(y)^2\> dy,
\end{equation}
where, by (\ref{weakconvfrom1}), the r.h.s.~converges for $N\to\infty$ to $\varepsilon R$. (\ref{estimate-error-n}),
(\ref{weakconvfrom1}), and the triangle inequality now readily lead to \eqref{convint1inffn}.

        We finally check \eqref{convint01fn}. We recall that
        \begin{align*}
        \lim_{N\to\infty}f_N(y)^2h_N(y)=\frac{(f(y))^2}{y}  \quad\text{for}\quad y\in[0,1[.
        \end{align*}
 Notice that this formula also holds for $y=0$, as $f$ is analytic in $0$ with $f(0)=0$.
Moreover, 
         \eqref{hNestimates}, the fact that $h_N(y)\leq N^{1/3}$, and Lemma~\ref{lemmalocalconvN13} show that
 for $N$ sufficiently large 
        \begin{align*}
        |f_N(y)^2h_N(y)|&\leq|f_N(y)^2-f(y)^2|h_N(y)+f(y)^2h_N(y)\\
        &\leq (f_N(y)+f(y))|f_N(y)-f(y)|N^\frac{1}{3}+\frac{f(y)^2}{y}\\
        &\leq (1+2f(y))O(1)+\frac{f(y)^2}{y}.
        \end{align*}
  As this is a  bounded continuous function for $y\in[0,1]$, we conclude from  dominated convergence that
\eqref{convint01fn} holds. This completes the proof.
\end{proof}

If we combine Theorem \ref{corrsigmaasymptotics}  with Theorem \ref{clt-main-a-general-x}, we obtain the following result,
which was also proved recently in a completely different way by Gorin and Kleptsyn [GK].

\begin{theorem}\label{CLT-Hermite-Ninfty}
Consider  the  Bessel processes 
$$(X_{t,k}^N)_{t\geq 0}=(X_{t,k,1}^N,\ldots,X_{t,k,N}^N )_{t\ge0}$$ of type $A_{N-1}$ on $C_N^A$ with start in $0\in C_N^A$.
Then, for each $t>0$, 
\begin{equation}\label{CLT-max-limit-main-Hermite}
        \lim_{N\to\infty}\left(\lim_{k\to\infty}N^\frac{1}{6}\sqrt{2k}\left(\frac{X_{t,k,N}^N}{\sqrt{2kt}}-z_{N,N}\right)\right)=
G
        \end{equation}
in distribution with some $\mathcal{N}(0,\sigma_{max}^2)$-distributed random variable $G$  with variance 
\begin{equation}\label{variance-final}
        \sigma_{max}^2:=\int_{0}^{\infty}\frac{\mathsf {Ai}(x+a_1)^2}{(\mathsf
{Ai}'(a_1))^2x}dx =0.834...
        \end{equation}
\end{theorem}

\begin{remarks} 
\begin{enumerate}
\item[\rm{(1)}] If we combine Theorem \ref{CLT-Hermite-Ninfty} with the formula  of Plancherel-Rotach
\begin{align*}
        \frac{z_{N,N}}{\sqrt{2N}}=1-\frac{|a_1|}{2N^\frac{2}{3}}+ r_N \quad\text{with}\quad  r_N=O(N^{-1}),
\end{align*}
we can state (\ref{CLT-max-limit-main-Hermite}) as 
\begin{equation}\label{CLT-max-limit-main-Hermite-2}
        \lim_{N\to\infty}\left(\lim_{k\to\infty}\left(N^\frac{2}{3}\left(\frac{X_{t,k,N}^N}{\sqrt{tN}}-2\sqrt k\right)
+2\sqrt k (|a_1|- N^\frac{2}{3}r_N)\right)\right)
=G.
        \end{equation}
Please notice that in this limit the term $2\sqrt kN^\frac{2}{3}r_N$ cannot be neglected.
\item[\rm{(2)}] Theorem \ref{CLT-Hermite-Ninfty} was stated by Dumitriu and Edelman (Corollary 3.4 in \cite{DE2}), where
the numerical value of $\sigma_{max}^2$ contains a misprint and the proof is sketched only. Moreover, the proof in  \cite{DE2} is based on 
    the representation of the covariance matrix $\Sigma_N$ in Theorem    \ref{clt-main-a-DE}.
This representation of  $\Sigma_N$ with essentially the same proof as above leads also to Theorem \ref{CLT-Hermite-Ninfty}
where then with the aid of (\ref{eqintf2}) one obtains 
\begin{equation}\label{variance-DE}
\sigma_{max}^2
        =2\frac{\int_{0}^{\infty}\mathsf {Ai}^4(x+a_1)dx}{\left(\int_{0}^{\infty}\mathsf
{Ai}^2(x+a_1)dx\right)^2}=2\int_{0}^{\infty}\left(\frac{\mathsf
{Ai}(x+a_1)}{\mathsf {Ai}'(a_1)}\right)^4dx.
        \end{equation}
A numerical computation shows that the value of (\ref{variance-DE}) seems to be equal to that in (\ref{variance-final}).
Unfortunately, we are not able to verify this equality in an analytic way, as this identity does not fit to 
known identities for the Airy function as in, for instance, \cite{VS}.
We  return to this point in the end of this section.

\item[\rm{(3)}] In \cite{RRV}, Ramirez, Rider, and Virag study the largest eigenvalues of $\beta$-Hermite ensembles where they first
take the limit $N\to\infty$ and then $\beta\to\infty$, i.e., $k\to\infty$ here. From the results in   \cite{RRV} one obtains that
\begin{equation}\label{CLT-max-limit-main-Hermite-3}
        \lim_{k\to\infty}\left(\lim_{N\to\infty}\left(N^\frac{2}{3}\left(\frac{X_{t,k,N}^N}{\sqrt{tN}}-2\sqrt k\right)
+2\sqrt k |a_1|\right)\right)
=G.
        \end{equation}
in distribution where $G$ is  $\mathcal{N}(0,\sigma_{max}^2)$-distributed with $\sigma_{max}^2$ as in (\ref{variance-DE}).
\end{enumerate}
\end{remarks}

\begin{remark} Clearly, the preceding limit results for the largest particle in the Hermite case can be 
transfered to the smallest particle by symmetry. 
\end{remark}

Next, we use the following Plancherel-Rotach formula
\begin{equation}\label{Planch-rot-r}
\frac{z_{N-r+1,N}}{\sqrt{2N}}=1-\frac{|a_r|}{2N^\frac{2}{3}}+O(N^{-1}),
\end{equation}
where $a_r$ is the $r$-th largest zero of the Airy function; this formula is derived from the well-known relationship between Hermite and Laguerre polynomials \cite{NIST}
\begin{align*}
H_{2n}(x)&=(-1)^n 2^{2n}n!L_n^{(-1/2)}(x^2),\\
H_{2n+1}(x)&=(-1)^n 2^{2n+1}n!L_n^{(1/2)}(x^2),
\end{align*}
and a corresponding Plancherel-Rotach formula for the Laguerre zeros given by (5) in Tricomi \cite{T}. This leads to the following
result for the $r$-th largest particle which can be also found in \cite{GK}:

\begin{theorem}\label{localconvfnrth}
	 For $r\in\mathbb{N}$ consider the functions 
	\begin{align*}
	f_N(y):=N^\frac{1}{6}Q_{\lfloor N^\frac{1}{3}y\rfloor,N}(z_{N-r+1,N}) \quad\text{for}\quad y\in[0, N^\frac{2}{3}[
	\end{align*}
	and  $f_N(y)=0$ otherwise.
	Then $(f_N)_{N\ge1}$ tends for $N\to\infty$ locally uniformly to 
	\begin{align*}
	f(y):=\frac{\mathsf {Ai}(y+a_r)}{\mathsf {Ai}'(a_r)}\quad\quad\text{for}\quad y\in[0,\infty[.
	\end{align*}
	Moreover,  the covariance matrices $\Sigma_{N}=:\left(\sigma_{i,j}\right)_{i,j=1,\ldots,N}$ of 
the  freezing	$\beta$-Hermite ensembles satisfy
	\begin{align*}
	\lim\limits_{N\rightarrow\infty}{N^\frac{1}{3}}{\sigma_{N-r+1,N-r+1}}=
	\sigma^2_{max,r},
	\end{align*}
	with $\sigma^2_{max,r}$ as specified in Theorem~\ref{localconvfnrth-introduction}.
\end{theorem}

\begin{proof}
	The proof is essentially the same as for Theorems \ref{localconvfn} and  \ref{corrsigmaasymptotics}
 where now $f_N$ and $f$ are now those of  Theorem \ref{localconvfnrth}.
 In particular, for $f$ we now have
  $$ f''(y)=(y+a_r)f(y) \quad\text{with}\quad f(0)=0,\quad f^\prime(0)=1.$$
	Moreover, $a_1$ has to be replaced  by $a_r$, and (\ref{Plancherel-Rotach}) by (\ref{Planch-rot-r}).
	We notice that now $f_N(t)\geq 0$ for $t\geq0$ does not hold; however we can still estimate 
	\begin{align*}
\left|	\int_{0}^{y}f_N(t)dt\right|\leq \int_{0}^{y}|f_N(t)|dt
	\end{align*} 
	with the triangle inequality. We then can proceed precisely as in (\ref{intfnlocalboundconcrete}).
\end{proof}


\begin{remark}
	For the first few values of $r$, the integral
	\begin{align*}
	\int_{0}^{\infty}\frac{\mathsf {Ai}(x+a_r)^2}{\mathsf {Ai}'(a_r)^2x}dx=\int_{a_r}^{\infty}\frac{\mathsf {Ai}(x)^2}{\mathsf {Ai}'(a_r)^2(x-a_r)}dx
	\end{align*}
	seems to be decreasing in $r$. This is indeed the case as $r\to\infty$. For this we
        decompose the last integral into the regions $[a_r,a_{r-1}[$, $[a_{r-1},a_1[$ and $[a_1,\infty[$, and estimate it in each case.
                    Let us first note that by \eqref{airy-der-at-zero},
	\begin{align*}
	\mathsf{Ai}'(a_r)^{-2}=\Big(\frac{2\pi^2}{3}\Big)^{1/3}r^{-1/3}+O(r^{-4/3}).
	\end{align*}
	Now, for the first region we  use  \eqref{airy-minus-inf}, \eqref{airy-zero-asymtot},
        the substitution $y=2(-x)^{3/2}/3\pi-r+3/4$, and  obtain for  $r\to\infty$ that
	\begin{align*}
	  \int_{a_r}^{a_{r-1}}\frac{\mathsf {Ai}(x)^2}{(x-a_r)}dx&=
          \Big(\frac{2}{3\pi}\Big)^{4/3}r^{-1/3}\int_{-1/2}^{1/2}\frac{3\cos(\pi y)^2}{1-2y}dy+O(r^{-4/3}).
	\end{align*} 
	Because $|\mathsf {Ai}(x)|<1$, \eqref{airy-zero-asymtot} leads to 
	\begin{align*}
	\int_{a_{r-1}}^{a_1}\frac{\mathsf {Ai}(x)^2}{(x-a_r)}dx\leq \log\frac{a_1-a_r}{a_{r-1}-a_r}&=\log\Big(r\Big(1+\Big(\frac{2}{3\pi}\Big)^{2/3}a_1r^{-2/3}+O(r^{-5/3})\Big)\Big)\\
	&=\log r+O(r^{-2/3}).
	\end{align*}
	Finally, Theorem~\ref{corrsigmaasymptotics} and  $a_r<a_1$ yield the bound
	\begin{align*}
		\int_{a_1}^\infty\frac{\mathsf {Ai}(x)^2}{(x-a_r)}dx\leq\int_{a_1}^\infty\frac{\mathsf {Ai}(x)^2}{(x-a_1)}dx<1.
	\end{align*}
	Putting everything together we see that for a sufficiently large $r$ there exists a constant $C>0$ such that
	\begin{align*}
	\int_{0}^{\infty}\frac{\mathsf {Ai}(x+a_r)^2}{\mathsf {Ai}'(a_r)^2x}dx\leq Cr^{-1/3}\log r.
	\end{align*}
	This stresses the fact that  $r\to\infty$ means that we go from the edge into the bulk, where repulsion interactions are stronger, or in other words, all variances there are much smaller than at the edge.
\end{remark}

\begin{remark}
  The proof of \cite{DE2} for the Airy asymptotics of the largest particle can  be extended to the $r$-th largest particle by using 
the form of the  covariance matrices in Theorem    \ref{clt-main-a-DE}.
  If we compare the result in this limit with that from Theorem \ref{localconvfnrth}, we obtain readily the following
identity for the Airy function for the $r$-th largest zero which seems to be unknown:
\begin{align*}
	2\int_{0}^{\infty}\mathsf
		{Ai}(x+a_r)^4dx=\mathsf
		{Ai}'(a_r)^2 \cdot\int_{0}^{\infty}\frac{\mathsf {Ai}(x+a_r)^2}{x}dx.
\end{align*}
\end{remark}

\section{Limit results for the largest  eigenvalue  in the Laguerre case}

We now discuss  the soft edge statistics in the  freezing Laguerre case 
similar to Section \ref{chapterhermitesoftedge} for the Hermite case.
This means that we analyze the limit behaviour of the largest eigenvalue in the freezing regime in Theorems \ref{clt-main-b}
and \ref{ev-b}
for $N\to\infty$ on the basis of Theorem \ref{covariance-matrix-general}.
We here again  use the ordered zeros  $z_{1,N}^{(\alpha)}<...,z^{(\alpha)}_{N,N}$  of the $N$-th Laguerre polynomial $L_N^{(\alpha)}$ as in Section 3.
Moreover, for each $N$, let
$(Q^{(\alpha)}_{k,N})_{k=0,..,N-1}$ be the dual polynomials associated with $(L_k^{(\alpha)})_{k=0,...,N}$  normalized as in 
(\ref{eq:Laguerrenormalization}). This means that the matrices
$$T_N:= (\sqrt{z_{i,N}^{(\alpha)}}Q^{(\alpha)}_{j-1,N}(z_{i,N}^{(\alpha)}))_{i,j=1,\ldots,N}$$
 are orthogonal 
with $T_N^T\Sigma_N T_N=\operatorname{diag}(\frac{1}{2},\frac{1}{4},...,\frac{1}{2N})$
 as in  the proof of Theorem \ref{covariance-matrix-general}.
The $Q_{k,N}^{(\alpha)}$ have the three-term-recurrence
\begin{align}\label{3termlaguerredualnorm}
xQ_{k,N}^{(\alpha)}(x)=&\sqrt{(N-k)(N-k+\alpha)}Q_{k-1,N}^{(\alpha)}(x)+(2(N-k)+\alpha-1)Q^{(\alpha)}_{k,N}\notag\\ 
+&\sqrt{(N-k-1)(N-k-1+\alpha)}Q_{k+1,N}^{(\alpha)}(x)\quad(k\le N)
\end{align}
with the initial conditions $Q^{(\alpha)}_{-1,N}=0$ and $Q^{(\alpha)}_{0,N}=\frac{1}{\sqrt{N(N+\alpha)}}$.

In the Laguerre case we have the following analog of  Theorem \ref{localconvfn}: 

\begin{theorem}\label{localconvfnlaguerre}
	Let $\alpha>-1 $ and define  the functions 
	\begin{align*}
	f_N(y):=N^\frac{1}{6}\sqrt{z^{(\alpha)}_{N,N}}Q_{\lfloor N^\frac{1}{3}y\rfloor,N}(z^{(\alpha)}_{N,N}) \quad\text{for}\quad y\in[0, N^\frac{2}{3}[
	\end{align*}
	and  $f_N(y)=0$ otherwise.
	Then $(f_N)_{N\ge1}$ tends for $N\to\infty$ locally uniformly  to
	\begin{align}\label{defflaguerre}
	f(y)=\frac{2^{\frac{1}{3}}\mathsf {Ai}(2^{\frac{2}{3}}y+a_1)}{\mathsf {Ai}'(a_1)}\quad(y\in[0,\infty[).
	\end{align}
\end{theorem}
\begin{proof}
	We put
	$q_k^{(\alpha)}:=Q^{(\alpha)}_{k,N}(z^{(\alpha)}_{N,N})$ for $k=0,\ldots, N-1$.
The Landau symbol $O$ will be always used for $N\rightarrow\infty$ and will be locally uniform w.r.t. $y\in[0,\infty[$.

The sharp Plancherel-Rotach  formula  for the zeros $z^{(\alpha)}_{N,N}$ in Theorem 1.2 of  \cite{G}  yields
	\begin{equation}\label{help-lag1}
	\frac{z^{(\alpha)}_{N,N}}{4N}=1+\frac{a_1}{(2N)^\frac{2}{3}}+O(N^{-1}).
	\end{equation}
	We will also use the Taylor expansion
	\begin{equation}\label{help-lag2}
		\sqrt{1+\frac{\alpha}{N-k}}=1+\frac{\alpha}{2(N-k)}+O(N^{-2}) 
\quad\quad \text{for} \quad 0\leq k\leq yN^{\frac{1}{3}}.
	\end{equation}
The recurrence relation (\ref{3termlaguerredualnorm}) for $x=z^{(\alpha)}_{N,N}$ and  a division by  $N$ yield
	\begin{align*}
&\left(\frac{z^{(\alpha)}_{N,N}}{N}-2\left(1-\frac{k}{N}\right)+\frac{\alpha-1}{N}\right)q_{k}^{(\alpha)}=\\
&\left(1-\frac{k+1}{N}\right)\sqrt{1+\frac{\alpha}{N-k-1}}q_{k+1}^{(\alpha)}
+\left(1-\frac{k}{N}\right)\sqrt{1+\frac{\alpha}{N-k}}q_{k-1}^{(\alpha)}.
	\end{align*}
	Using \eqref{help-lag1} and (\ref{help-lag2}), we obtain
	\begin{align*}
	&q_{k}^{(\alpha)}\left(2+\frac{2^{\frac{4}{3}}a_1}{N^{\frac{2}{3}}}+\frac{2k}{N}+O(N^{-1})\right)\\
	=&q_{k-1}^{(\alpha)}\left(1-\frac{k}{N}+\frac{\alpha}{2(N-k)}(1-\frac{k}{N})+O(N^{-1})(1-\frac{k}{N})\right)\\
	&+q_{k+1}^{(\alpha)}\left(1-\frac{k+1}{N}+\frac{\alpha}{2(N-k-1)}(1-\frac{k+1}{N})+O(N^{-1})(1-\frac{k+1}{N})\right)\\
	=&q_{k-1}^{(\alpha)}\left(1-\frac{k}{N}+O(N^{-1})\right)+
	q_{k+1}^{(\alpha)}\left(1-\frac{k+1}{N}+O(N^{-1})\right).
	\end{align*}
Hence,
	\begin{align}\label{short-rec-lagu}
		q_{k+1}^{(\alpha)}+q_{k-1}^{(\alpha)}-2q_{k}^{(\alpha)}=\frac{k+1}{N}q_{k+1}^{(\alpha)}+\frac{k}{N}q_{k-1}^{(\alpha)}+
		\left(\frac{2^{\frac{4}{3}}a_1}{N^{\frac{2}{3}}}+\frac{2k}{N}\right)q_{k}^{(\alpha)}
		\\+O(N^{-1})q_{k+1}^{(\alpha)}+O(N^{-1})q_{k}^{(\alpha)}+O(N^{-1})q_{k-1}^{(\alpha)}.
	\notag\end{align}
	Eq.~(\ref{short-rec-lagu}) is very similar to Eq.~(\ref{eq3termapproximated}), so we skip some details as the calculation below will be very similar to Section 5.
We sum (\ref{short-rec-lagu}) over $k=0,...,l$ and then over $l=0,...,\lfloor y N^{\frac{1}{3}}\rfloor-1$.
After multiplying by $N^{\frac{1}{6}}\sqrt{z^{(\alpha)}_{N,N}}$ we obtain from (\ref{help-lag1})  that the LHS is equal to
	\begin{align*}
		f_N(y)-\frac{(1+\lfloor y N^{\frac{1}{3}}\rfloor) N^{\frac{1}{6}}\sqrt{z^{(\alpha)}_{N,N}}}{\sqrt{N(N+\alpha)}}
		&=f_N(y)-\frac{(1+\lfloor y N^{\frac{1}{3}}\rfloor) N^{\frac{1}{6}}2\sqrt{N}(1+O(N^{-2/3}))}{\sqrt{N(N+\alpha)}}\\
		&=f_N(y)-2y+O(N^{-1/3})
	\end{align*}
	and the RHS to
	\begin{align}\label{eqdiscreteintlaguerre}
		\sum_{l=0}^{\lfloor
			yN^\frac{1}{3}\rfloor-1}\sum_{k=0}^{l}\Bigg(
		&\frac{k+1}{N}N^{\frac{1}{6}}\sqrt{z^{(\alpha)}_{N,N}}q_{k+1}^{(\alpha)}
		+\frac{k}{N}N^{\frac{1}{6}}\sqrt{z^{(\alpha)}_{N,N}}q_{k-1}^{(\alpha)}\notag\\
		&+N^{\frac{1}{6}}\sqrt{z^{(\alpha)}_{N,N}}\bigg(\frac{2^{\frac{4}{3}}a_1}{N^{\frac{2}{3}}}+\frac{2k}{N}\bigg)q_{k}^{(\alpha)}
		\Bigg)+O(N^{-1/3}).
	\end{align}
Note that in the second case, we used the estimation
	\begin{align*}
			\sum_{l=0}^{\lfloor
			yN^\frac{1}{3}\rfloor-1}\sum_{k=0}^{l}\left(O(N^{-1})q_{k+1}^{(\alpha)}+O(N^{-1})q_{k}^{(\alpha)}+O(N^{-1})q_{k-1}^{(\alpha)}\right)=O(N^{-\frac{1}{3}})
	\end{align*}
which can be proved precisely as Eq.~(\ref{estimates-rho}) in the proof of  Lemma \ref{lemmaerrbound}.
We next use the index shifts  $k+1\mapsto k$ and  $k-1\mapsto k$ 
in (\ref{eqdiscreteintlaguerre}) and obtain  that the r.h.s. above is
	\begin{align*}
			&\frac{1}{N^{\frac{1}{3}}}\sum_{l=0}^{\lfloor
			yN^\frac{1}{3}\rfloor-1}\frac{1}{N^{\frac{1}{3}}}\sum_{k=0}^{l}\Bigg(
		N^{\frac{1}{6}}\sqrt{z^{(\alpha)}_{N,N}}q_{k}^{(\alpha)}\bigg(\frac{k+1}{N^{\frac{1}{3}}}+\frac{k}{N^{\frac{1}{3}}}+{2^{\frac{4}{3}}a_1}{}+\frac{2k}{N^{\frac{1}{3}}}\bigg)\Bigg)\\
		+&\frac{1}{N^{\frac{2}{3}}}\sum_{k=0}^{\lfloor
			yN^\frac{1}{3}\rfloor}\frac{k}{N^{\frac{1}{3}}}N^{\frac{1}{6}}\sqrt{z^{(\alpha)}_{N,N}}q_{k}^{(\alpha)}
		-\frac{1}{N^{\frac{2}{3}}}\sum_{k=0}^{\lfloor
			yN^\frac{1}{3}\rfloor-1}\frac{k+1}{N^{\frac{1}{3}}}N^{\frac{1}{6}}\sqrt{z^{(\alpha)}_{N,N}}q_{k}^{(\alpha)}\\
		=&\frac{1}{N^{\frac{1}{3}}}\sum_{l=0}^{\lfloor
			yN^\frac{1}{3}\rfloor-1}\frac{1}{N^{\frac{1}{3}}}\sum_{k=0}^{l}\Bigg(
		N^{\frac{1}{6}}\sqrt{z^{(\alpha)}_{N,N}}q_{k}^{(\alpha)}\bigg(\frac{k+1}{N^{\frac{1}{3}}}+\frac{k}{N^{\frac{1}{3}}}+{2^{\frac{4}{3}}a_1}{}+\frac{2k}{N^{\frac{1}{3}}}\bigg)\Bigg)+O(N^{-\frac{1}{3}}),
	\end{align*}
	where  for the last equation an analog estimation to that in (\ref{intfnlocalbound}) has been used.

In summary we have proved that
	\begin{align*}
		f_N(y)-2y+O(N^{-{1}/{3}})=\frac{1}{N^{\frac{2}{3}}}\sum_{l=0}^{\lfloor
			yN^\frac{1}{3}\rfloor-1}\sum_{k=0}^{l}\Bigg(
		N^{\frac{1}{6}}\sqrt{z^{(\alpha)}_{N,N}}q_{k}^{(\alpha)}\bigg(\frac{4k+1}{N^{\frac{1}{3}}}+2^{\frac{4}{3}}a_1\bigg)\Bigg)
	\end{align*}
If we use (\ref{discretedoubleint}), we see that this  leads to the integral equation
	\begin{align*}
	f_N(y)=\int_{0}^{y}\int_{0}^{s}f_N(s)(4t+2^{\frac{4}{3}}a_1)dtds+2y+O(N^{-\frac{1}{3}}).
	\end{align*}
As the function $f$ defined in (\ref{defflaguerre}) satisfies
	$f(0)=0$, $f'(0)=2$ and $f''(x)=(4x+2^{\frac{4}{3}}a_1)f(x)$, we obtain from
	the Lemma of Gronwall (see also Lemma \ref{lemmalocalconvN13}) that 
	\begin{align*}
		|f_N(y)-f(y)|=O(N^{-{1}/{3}}).
	\end{align*}
\end{proof}

We now apply Theorem \ref{localconvfnlaguerre} to the $(N, N )$-entries of the covariance matrices $\Sigma_N$
for $\beta$-Laguerre ensembles in the freezing regime in Theorem 4.8 for $N\rightarrow\infty$.

\begin{corollary}
	 Consider the  covariance matrices $\Sigma_{N}=:\left(\sigma_{i,j}\right)_{i,j=1,...,N}$ of 
	\text{$\beta$-Laguerre} ensembles in the freezing regime. Then
	\begin{align*}
	\lim\limits_{N\rightarrow\infty}{N^\frac{1}{3}}{\sigma_{N,N}}=
	\frac{1}{2}\int_{0}^{\infty}\frac{\mathsf {Ai}(x+a_1)^2}{\mathsf {Ai}'(a_1)^2x}dx =0.417\ldots
	\end{align*}

\end{corollary}
\begin{proof}
	The proof is completely analog to the proof of Theorem \ref{corrsigmaasymptotics}.
\end{proof}

In summary, we obtain:

\begin{theorem}
	Consider the  Bessel processes $(X_{t,k})_{t\ge0}$ of type $B_{N}$ on $C_N^B$ for
	$k=(k_1,k_2)=(\kappa\cdot(\alpha+1),\kappa)$
	with $\kappa>0$, $\alpha>-1$ with start in $0\in C_N^B$.
	Then, for each $t>0$,
	$$\lim_{N\to \infty}\left(\lim_{\kappa\to\infty} N^{\frac{1}{6}}\sqrt{2\kappa}\left(\frac{X_{t,(\kappa\cdot\alpha+\kappa,\kappa)}}{\sqrt{2\kappa t} } -  \sqrt{ z_{N,N}^{(\alpha)}}\right)\right)=G$$
	in distribution with some $\mathcal{N}(0,\frac{1}{2}\sigma_{max}^2)$-distributed random variable $G$ with  variance
	\begin{align*}
		\frac{1}{2} \sigma_{max}^2:=\frac{1}{2}\int_{0}^{\infty}\frac{\mathsf {Ai}(x+a_1)^2}{(\mathsf
			{Ai}'(a_1))^2x}dx =0.417...
	\end{align*}

\end{theorem}

{\bf Acknowledgement:} The first author has been supported by
 the Deutsche Forschungsgemeinschaft
 (DFG) via RTG 2131 \textit{High-dimensional Phenomena in Probability - Fluctuations
and Discontinuity}
 to visit Dortmund for the preparation of this paper, as well as by JSPS KAKENHI Grant Number JP19K14617. The second author has been completely supported by this RTG  2131.

\end{document}